\documentclass[11pt,reqno]{amsart}
\usepackage{amsthm,amssymb,amsmath}
\usepackage{xcolor}
\usepackage[foot]{amsaddr}

\newtheorem{theorem}{Theorem}
\newtheorem*{theorem*}{Theorem}
\newtheorem{lemma}{Lemma}
\newtheorem{corollary}{Corollary}

\theoremstyle{definition}

\newtheorem{definition}{\sc Definition}
\newtheorem*{definition*}{\sc Definition}

\newtheorem*{remark*}{\sc Remark}

\newtheorem*{remarks*}{\sc Remarks}

\newtheorem*{example*}{\sc Example}

\newcommand{\loc}{{\rm loc}}

\newcommand{\Real}{{\rm Re}}
\newcommand{\Imag}{{\rm Im}}

\newcommand{\clos}{{\rm clos}}

\newcommand{\LL}{\text{\it \L}}

\begin{document}

\title{Regularity of solutions to Kolmogorov equation with Gilbarg-Serrin matrix}

\author{D.\,Kinzebulatov}

\address{Universit\'{e} Laval, D\'{e}partement de math\'{e}matiques et de statistique, 1045 av.\,de la M\'{e}decine, Qu\'{e}bec, QC, G1V 0A6, Canada}

\email{damir.kinzebulatov@mat.ulaval.ca}

\author{Yu.\,A.\,Sem\"{e}nov}

\address{University of Toronto, Department of Mathematics, 40 St.\,George Str, Toronto, ON, M5S 2E4, Canada}

\email{semenov.yu.a@gmail.com}

\begin{abstract}
 In $\mathbb R^d$, $d \geq 3$, consider the divergence and the non-divergence form operators 
\begin{equation}
\label{div0}
\tag{$i$}
-\Delta - \nabla \cdot (a-I) \cdot \nabla + b \cdot \nabla, 
\end{equation}
\begin{equation}
\label{nondiv0}
\tag{$ii$}
 - \Delta - (a-I) \cdot \nabla^2 + b \cdot \nabla,
\end{equation}
where the second order perturbations are given by the matrix $$a-I=c|x|^{-2}x \otimes  x, \quad c>-1.$$ The vector field $b:\mathbb R^d \rightarrow \mathbb R^d$ is form-bounded with the form-bound $\delta>0$ (this includes $[L^d + L^\infty]^d$, as well as vector fields having critical-order singularities).
We characterize quantitative dependence on $c$ and $\delta$ of the $L^q \rightarrow W^{1,qd/(d-2)}$ regularity of the resolvents of the operator realizations of \eqref{div0}, \eqref{nondiv0}  in $L^q$, $q \geq 2 \vee ( d-2)$ as (minus) generators of positivity preserving $L^\infty$ contraction $C_0$ semigroups.

\end{abstract}


\thanks{D.K.\,is supported by the Natural Sciences and Engineering Research Council of Canada and the Fonds de recherche du Qu\'{e}bec -- Nature et technologies}


\maketitle

In this paper we are concerned with the second-order perturbations of $-\Delta$,
\begin{equation}
\label{div_}
\begin{array}{c}
-\Delta - \nabla \cdot (a-I) \cdot \nabla, \\
-\Delta - (a-I) \cdot \nabla^2, 
\end{array}
\end{equation}
\begin{equation*}
a_{ij}(x):=\delta_{ij}+c|x|^{-2}x_ix_j, \quad c>-1, c \neq 0.
\end{equation*}
These are model examples of divergence/non-divergence form operators that are not accessible by classical means such as the parametrix \cite{F}, \cite[Ch.\,IV]{LSU}. Although the matrix $a$ is discontinuous at $x=0$, it is uniformly elliptic,
so by the De Giorgi-Nash theory, the solution $u \in W^{1,2}(\mathbb R^d)$ to the  elliptic equation $(\mu -\nabla \cdot a \cdot \nabla) u = f$, $\mu>0$, $f \in L^p\cap L^2$, $p \in ]\frac{d}{2},\infty[$, is in $C^{0,\gamma}$, where the H\"{o}lder continuity exponent $\gamma \in ]0,1[$ depends only on $d$ and $c$. 
The operators \eqref{div_} and their modifications have been studied by many authors  in order to precise 
the relationship between the regularity properties of the solution and the continuity properties of the matrix, see \cite{GS}, \cite{Me}, \cite[Ch.\,1.2]{LU}, \cite{ABT}, \cite{MSS,MSS2}, \cite{OG}, \cite{A} and references therein. In fact, there is a quantitative dependence of the regularity properties of solutions of the corresponding parabolic and elliptic equations on the value of $c$ (see the cited papers, see also the results below). In this sense, the matrix $a$ has a critical-order discontinuity at the origin.

The critical-order perturbations of $-\Delta$ and its generalizations have been the subject of intense study over the past few decades as they reveal otherwise inaccessible aspects of the theory of the unperturbed operator.
For example, consider the operator $-\Delta - V$, $V(x)=\delta \frac{(d-2)^2}{4}|x|^{-2}$, on $\mathbb R^d$, $d \geq 3$. If $0 <\delta < 1$, then the self-adjoint operator realization $H^-$ of $-\Delta - V$ on $L^2 \equiv L^2(\mathbb R^d)$  is defined  as the
generator of a $C_0$ semigroup $e^{-tH^-}=L^2{\mbox-}s{\mbox-}\lim e^{-tH^-(V_n)}$, $V_n=V \wedge n$. 
For $\delta>1$, however, by the celebrated result of \cite{BG} (see also \cite{GZ}), 
$$
\lim_n e^{-tH^-(V_n)}u_0(x)=\infty, \quad t>0, \quad x \in \mathbb R^d, \quad u_0 \geq 0,\; u_0 \not\equiv 0,
$$
i.e.\,all positive solutions blow up instantly at any point. This phenomenon is not observable on $V=\delta V_0$ for any $V_0 \in L^{\frac{d}{2}}$, regardless of how large $\delta>0$ is
(in this sense, the class $L^{\frac{d}{2}}$ does not contain potentials having critical-order singularities). 
The perturbations $\nabla \cdot (a-I) \cdot \nabla$, $(a-I)\cdot\nabla^2$, $a-I=c|x|^{-2}x \otimes x$, of $-\Delta$ can be viewed as the second-order analogues of the critical potential $V(x)=\delta\frac{(d-2)^2}{4}|x|^{-2}$.

Our goal is to determine to what extent adding $\nabla \cdot (a-I) \cdot \nabla$, $(a-I)\cdot\nabla^2$ affects the perturbation-theoretic and the regularity properties of $-\Delta$. Our interest is motivated by applications to diffusion processes, and so we restrict our study to the first-order perturbations.
 The following result concerning the regular case $c=0$ will serve as the point of departure.
Consider in $\mathbb R^d$, $d \geq 3$, the operator
$$
-\Delta + b \cdot \nabla, \quad b:\mathbb R^d \rightarrow \mathbb R^d.
$$
If $|b| \in L^r$, $r>d$, then its fundamental solution admits two-sided Gaussian bounds, and determines a $C_0$ semigroup  on $L^p$ for any $p \in [1,\infty[$. Its generator $\Lambda_p$ is an operator realization of $-\Delta + b \cdot \nabla$ in $L^p$.
For $p \in [1,r]$, $D(\Lambda_p)=(1-\Delta)^{-1}L^p~(=D(-\Delta) \text{ in } L^p)$. However, for $p \in ]r,\infty[$, $D(\Lambda_p)$ no longer coincides with $(1-\Delta)^{-1}L^p$, and there is no direct connection between $\Lambda_p$ and the algebraic sum $-\Delta + b \cdot \nabla$.

The class $|b| \in L^r$, $r>d$, is contained in a much larger class of the form-bounded vector fields:

\begin{definition}
A measurable vector field $b:\mathbb R^d \rightarrow \mathbb R^d$ is said to be form-bounded (write $b \in \mathbf{F}_\delta$, $\delta>0$)
if $|b| \in L^2_{\loc}$ and 
there exist a constant $\lambda=\lambda_\delta>0$ such that 
$$
\||b|(\lambda -\Delta)^{-\frac{1}{2}}\|_{2\rightarrow 2}\leq \sqrt{\delta}.
$$
(here and below, $\|\cdot\|_{p \rightarrow q}:=\|\cdot\|_{L^p \rightarrow L^q}$).
\end{definition}

The model vector field $b(x):=\sqrt{\delta}\frac{d-2}{2}|x|^{-2}x$, $0<\delta<1$, having critical-order singularity at $x=0$ is contained
in $\mathbf{F}_\delta$ by Hardy's inequality.
The class $\mathbf{F}_\delta$ contains $b \in [L^d + L^\infty]^d$, with $\delta>0$ that can be chosen arbitrarily small (by Sobolev's inequality). $\mathbf{F}_\delta$ also contains classes of vector fields having critical-order singularities, such as the weak $L^d$ class (by Strichartz' inequality \cite{KPS}), the Campanato-Morrey class or the Chang-Wilson-Wolff class \cite{CWW}, with $\delta$ depending on the norm of $|b|$ in these classes. It is clear that
$
b_1 \in \mathbf{F}_{\delta_1}$, $b_2 \in \mathbf{F}_{\delta_2}$ $\Rightarrow$ $ b_1+b_2 \in \mathbf{F}_\delta$, $ \sqrt{\delta}:=\sqrt{\delta_1}+\sqrt{\delta_2}.
$
See e.g.\,\cite{KiS} for further details concerning the class $\mathbf{F}_\delta$. The following is our point of departure:
If $\delta<1 \wedge \bigl(\frac{2}{d-2}\bigr)^2$, then for
$$
u=(\mu+\Lambda_q(b))^{-1}f, \quad \mu > \frac{\lambda \delta}{2(q-1)}, \quad f \in L^q, \quad q \in [2 \vee (d-2),\frac{2}{\sqrt{\delta}}[,
$$
where $\Lambda_q(b)$ is an operator realization of $-\Delta + b\cdot\nabla$
as the (minus) generator of a  positivity preserving $L^\infty$ contraction $C_0$ semigroup in $L^q$,
one has
\begin{equation}
\tag{$\ast$}
\label{reg_double_star_}
\begin{array}{c}
\|\nabla u \|_{q} \leq  K_1(\mu-\mu_0)^{-\frac{1}{2}} \|f\|_q, \\[2mm]
\|\nabla u \|_{\frac{qd}{d-2}} \leq  K_2 (\mu-\mu_0)^{\frac{1}{q}-\frac{1}{2}} \|f\|_q.
\end{array}
\end{equation}
where $K_i=K(\mu,\delta,q)$, $i=1,2$, see \cite[Lemma 5]{KS} (see also \cite[Theorems 3.7-3.10]{KiS}).
In particular, if $q>d-2$, then by the Sobolev Embedding Theorem $u \in C^{0,\gamma}$, $\gamma=1-\frac{d-2}{q}$. (Note that the dependence on $\mu$ in \eqref{reg_double_star_} is the same as one would have for $u=(\mu-\Delta)^{-1}f$.)
The second estimate in \eqref{reg_double_star_} and the iteration method $L^p \rightarrow L^\infty$ in \cite{KS} (see also \cite[sect.\,3.6]{KiS}) allow to construct a Feller semigroup associated with $-\Delta + b \cdot \nabla$ on the space $C_\infty=\{f \in C(\mathbb R^d): \lim_{x \rightarrow \infty}f(x)=0\}$ (with the $\sup$-norm).

In Theorem \ref{thm2} below (the \textit{main result})  we show that the perturbation $-\nabla \cdot (a-I) \cdot \nabla$ of $-\Delta$ preserves, under appropriate assumptions on $c$, the properties of $-\Delta$ that allow to establish \eqref{reg_double_star_} for $u=\big(\mu+\Lambda_q(a,b)\big)^{-1}f$, where $\Lambda_q(a,b)$ is an operator realization of
$$
-\Delta - \nabla \cdot (a-I) \cdot \nabla + b \cdot \nabla, \qquad b \in \mathbf{F}_\delta
$$
as the (minus) generator of a  positivity preserving $L^\infty$ contraction $C_0$ semigroup in $L^q$. (To be precise, the latter is constructed as the limit of the semigroups corresponding to smooth approximations of $a$, $b$.) 
The existing results on $-\Delta - \nabla \cdot (a-I) \cdot \nabla + b \cdot \nabla$ provide a detailed regularity theory of this operator for  $b(x)=c|x|^{-2}x$, see \cite{MSS}, see also references therein  (here and below, we restrict our attention to the locally unbounded vector fields). In contrast to these results, our results follow from the a priori estimates for the approximating semigroups, and do not depend on the specific structure of  $b$
 (such as being differentiable or rotationally-symmetric).

Our method admits immediate extension to 
\begin{equation}
\label{a_gen}
a_{ij}(x)=\delta_{ij}+\sum_lc_l\kappa_{ij}(x - x^l), \quad \kappa_{ij}(x)=|x|^{-2}{x_ix_j},
\end{equation}
$$c_+:=\sum_{c_l>0}c_l<\infty ,\quad c_-:=\sum_{c_l<0}c_l >-1,$$
where $\{x^l\}$ is an arbitrary countable subset of $\mathbb R^d$, e.g.\,dense in $\mathbb R^d$. 
Our method does not depend on the geometry of the set of the points of discontinuity of the matrix \eqref{a_gen}.

Set 
$(\nabla a)_k := \sum_{i=1}^d (\partial_{x_i} a_{ik})$, $1 \leq k \leq d$. Then $\nabla a=c(d-1)|x|^{-2}x$ , and so $\nabla a \in \mathbf{F}_\delta$, $\delta=\frac{4c^2(d-1)^2}{(d-2)^2}$. 
The latter allows us to construct an operator realization   of the non-divergence form operator $$-a \cdot \nabla^2 + b \cdot \nabla \equiv -\sum_{i,j=1}^d a_{ij}(x) \partial_{x_i} \partial_{x_j} + \sum_{k=1}^d b_k(x) \partial_{x_k}, \quad b \in \mathbf{F}_{\delta_1}$$ 
in $L^q$ as $\Lambda_q(a,\nabla a + b)$
(formally, $-a \cdot \nabla^2 + b \cdot\nabla \equiv - \nabla \cdot a \cdot \nabla + (\nabla a) \cdot \nabla + b \cdot \nabla$) and then characterize 
quantitative  dependence of the $W^{1,p}$ regularity of $u \equiv (\mu+\Lambda_q(a,\nabla a + b))^{-1}f$, $f \in L^q$, on $c,d,q, \mu$ and $\delta_1$ (Theorems \ref{thm3} and \ref{thm4}). 
The class of 
of admissible first order perturbations of $-a \cdot \nabla^2$, i.e.\,$b \cdot \nabla$, $b \in \mathbf{F}_{\delta_1}$ (see also the previous remark) allows us to conclude that this result can not be achieved on the basis of the Krylov-Safonov a priori estimates \cite[Ch.\,4.2]{Kr}. (We note that the operator $-a\cdot\nabla^2$ with $\partial_{x_k}a_{ij} \in L^{d,\infty}$ has been studied earlier  in \cite{AT}; see also \cite{ABT}.)

The second estimate in \eqref{reg_double_star_}, and a variant of the iteration method in \cite{KS}, allow to construct $L^p$-strong Feller semigroups associated with $-\nabla\cdot a \cdot\nabla + b \cdot \nabla$ and $-\nabla\cdot a \cdot\nabla  + (\nabla a + b) \cdot \nabla$, $b \in \mathbf{F}_\delta$, in $C_\infty$. We plan to address this matter  in another paper.

Concerning the application of \eqref{reg_double_star_} to establishing the $C^{0,\gamma}$ continuity of $u$, we note the following. Let $d \geq 4$. In the proof of Theorem \ref{thm2} we establish a somewhat stronger than \eqref{reg_double_star} higher order derivatives estimate $$\|\nabla |\nabla u|^{\frac{q}{2}}\|_2^2 \leq K \|f\|^q_q.$$
It follows that 
 if $q>d-2$, then $u \in C^{0,\gamma}$, $\gamma=1-\frac{d-2}{q}$. We do not appeal, for the purpose of establishing $u \in C^{0,\gamma}$, to the $\|u\|_{2,r}$ estimates for a large $r$. In fact, it is not clear if such an estimate exists for a general $b \in \mathbf{F}_\delta$, $\delta>0$. (Below we establish the $W^{2,2}$ estimates on $u$, but in order to conclude the H\"{o}lder regularity of $u$ these are only sufficient in the dimension $d=3$.) We note that in \cite{MSS2} the authors construct an operator realization $A_p$ of $-\Delta + \nabla \cdot (a-I)\cdot \nabla + b \cdot \nabla$, $b(x)=\sqrt{\delta}\frac{d-2}{2}|x|^{-2}x$, $0<\delta<4$, in $L^p$, and completely characterize its domain, establishing, in particular, for $u \in D(A_p)$ that
$$\nabla_i\nabla_k u \in L^p, \quad i,k=1,\dots,d \quad \text{ if } p \in ]\frac{2}{2-\frac{d-2}{d}\sqrt{\delta}},\frac{d}{2}[.$$ However, the latter does not allow to conclude that $u \in C^{0,\gamma}$ for any $\gamma>0$. 

In view of the previous remark we note that having a complete characterization of the domain of (an operator realization of) $ -\nabla \cdot a \cdot \nabla$ in $L^q$ for some $q$ is not sufficient on its own in order to characterize regularity of the domain of $ -\nabla \cdot a \cdot \nabla + b \cdot \nabla$, $b \in \mathbf{F}_\delta$, in $L^q$ (as is already apparent in the case $a=I$ discussed above).  
In particular, Theorem \ref{thm2} below is by no means a consequence of Theorem \ref{thm1}.

In Theorems \ref{thm1}-\ref{thm4} below we have tried to find the least restrictive assumptions on $c$ and $\delta$, permitted by the method, such that the estimates \eqref{reg_est000}, \eqref{reg_double_star} hold (we note that our result is not of Cordes-type). The weaker result that there exist sufficiently small $c$ and $\delta$ such that the estimates \eqref{reg_est000}, \eqref{reg_double_star} are valid  (still not accessible by the existing results prior to our work)  can be obtained with considerably less effort by following the proof and discarding the corresponding multiples of $c$ and $\delta$.

The method of this paper is suited to treat classes of second-order perturbations $-\nabla \cdot (a-I)\cdot \nabla$, $-(a-I) \cdot \nabla^2$ of $-\Delta$ more general than  \eqref{a_gen}, for example, given by $a-I=v \otimes v$, where (bounded) $v:\mathbb R^d \rightarrow \mathbb R^d$, $v \in W^{1,2}_{\loc}(\mathbb R^d,\mathbb R^d)$ satisfies $\big(\sum_k (\nabla v_k)^2\big)^{\frac{1}{2}} \in \mathbf{F}_\delta$. We plan to address this matter in another paper.

The arguments in this paper can be transferred without significant changes from $\mathbb R^d$ to the ball $B(0,1)$.

We have included Appendix \ref{appendix_existence} to make the paper self-contained.

\medskip

\textbf{1.~}We now state our results in full. For reader's convenience, we start with the case $b=0$.

\begin{theorem}[$-\nabla \cdot a \cdot \nabla$]
\label{thm1} Let $d \geq 3$, $a(x)=I+c|x|^{-2}x \otimes x$, $c>-1$.

{\rm(\textit{i})} The formal differential operator $-\nabla \cdot a \cdot \nabla$ has an operator realization  $A_q$ on $L^q$, $q \in [1,\infty\big[$,  as the (minus) generator of a  positivity preserving $L^\infty$ contraction $C_0$ semigroup.

Set $u:=(\mu+A_q)^{-1}f$, $\mu > 0$, $f \in L^q$.

{\rm(\textit{ii})}  Let $d\geq 4$. Assume that  $q \geq d-2$ and 
$$
-\frac{(q-1)(d-2)^2}{q^2\ell_2}<c<\frac{(q-1)(d-2)^2}{q^2\ell_1},
$$
where
$$
\ell_1 \equiv \ell_1(q,d):=d-1 -d\frac{d-2}{q}+(1+\theta)\frac{(d-2)^2}{q^2}, \qquad \theta=\frac{1}{2(d-1)},
$$ 
$$
\ell_2 \equiv \ell_2(q,d):=d-1+(q-1)\frac{(d-2)^2}{q^2}.
$$
Then $u \in \bigcap_{q\leq p\leq \frac{qd}{d-2}} W^{1,p}$, and there exist constants $K_l=K_l(d,q,c)$, $l=1,2$, such that 
\begin{equation}
\tag{$\star$}
\label{reg_est000}
\begin{array}{c}
\|\nabla u \|_{q} \leq  K_1\mu^{-\frac{1}{2}} \|f\|_q, \\[2mm]
\|\nabla u \|_{\frac{qd}{d-2}} \leq  K_2 \mu^{\frac{1}{q}-\frac{1}{2}} \|f\|_q.
\end{array}
\end{equation}
The dependence on $q$ and $\mu$ in (\ref{reg_est000}) is the best possible. 

\smallskip

{\rm(\textit{iii})} Let $d\geq 3$. Assume that  
\[
-\bigg(1+\frac{4(d-1)}{(d-2)^2}\bigg)^{-1}<c<\frac{(d-2)^2}{4}.
\]
Then $u \in W^{2,2}$ and (\ref{reg_est000}) holds with $q=2$.
\end{theorem}

Of special interest are the minimal assumptions on $c$ such that the second estimate in \eqref{reg_est000} holds with \textit{some} $q>d-2$.

\begin{corollary}
\label{cor1}
For $d=3$ and $-\frac{1}{9}<c<\frac{1}{4}$,
\[
(\mu+A_2)^{-1}L^2 \subset C^{0,\gamma}, \quad \gamma = \frac{1}{2}.
\]
For all $d\geq4$, $-\frac{1}{2+\frac{2}{d-3}}<c<2(d-1)(d-3)$ and $q> d-2$ sufficiently close to $d-2$,  
\[
(\mu+A_q)^{-1}L^q \subset C^{0,\gamma}, \quad \gamma=1-\frac{d-2}{q}.
\]
\end{corollary}

\begin{theorem}[$-\nabla \cdot a \cdot \nabla + b \cdot \nabla$]
\label{thm2}
 Let $d \geq 3$, $a(x)=I+c|x|^{-2}x \otimes x$, $c>-1$, $b \in \mathbf{F}_\delta$.

{\rm(\textit{i})}  If $\delta_1:=[1 \vee (1+c)^{-2}]\,\delta < 4$, then $-\nabla \cdot a \cdot \nabla + b \cdot \nabla$ has an operator realization  $\Lambda_q(a,b)$ in $L^q$, $q \in \big[\frac{2}{2-\sqrt{\delta_1}}, \infty\big[$, as the (minus) generator of a  positivity preserving $L^\infty$ contraction $C_0$ semigroup.

{\rm(\textit{ii})} Let $d\geq 4$. Assume that $q \geq d-2$, $\delta<1 \wedge \frac{4}{(d-2)^2}$ and $c$ satisfy one of the following two conditions:
\smallskip

{\rm 1)} $c>0$ and  $1+c\big(1-\frac{1}{2(d-1)}-\frac{q\sqrt{\delta}}{4}\big) \geq 0$,
and $$(q-1)\frac{(d-2)^2}{q^2} - \LL_1(c,\delta)>0,$$

\smallskip

{\rm 2)} $-1<c<0$ and $1+c\big(1+\frac{q\sqrt{\delta}}{4}\big) \geq 0$,
and 
 $$(q-1)\frac{(d-2)^2}{q^2} -\LL_2(-c,\delta)>0,$$

\smallskip

\noindent where
\begin{align*}
\LL_{1}(c,\delta)\equiv \LL_{1}(c,\delta,q,d)&:=c\biggl[1+\frac{q-2}{q}(d-2)-\bigg(q-1-\frac{1}{2(d-1)}\bigg)\frac{(d-2)^2}{q^2} \biggr] \\
&+ \frac{c\sqrt{\delta}}{2}\biggl[\frac{(d-2)^2}{q}+(d+3)(d-2)\biggr] +\biggl[\frac{q^2\delta}{4}+(q-2)\frac{q\sqrt{\delta}}{2}\biggr]\frac{(d-2)^2}{q^2},
\end{align*} 
\begin{align*}
\LL_{2}(c,\delta) \equiv \LL_{2}(c,\delta,q,d)&:=c\biggl[-d+1+\frac{q-2}{q}(d-2)+(q-1)\frac{(d-2)^2}{q^2} \biggr] \\
&+ \frac{c\sqrt{\delta}}{2}\biggl[\frac{(d-2)^2}{q}+(d+3)(d-2)\biggr] +\biggl[\frac{q^2\delta}{4}+(q-2)\frac{q\sqrt{\delta}}{2}\biggr]\frac{(d-2)^2}{q^2}.
\end{align*}
Then there exist constants 
$\mu_0=\mu_0(d,q,c,\delta)>0$ and $K_l=K_l(d,q,c,\delta)$
, $l=1,2$, such that for all $\mu>\mu_0$, $u:=(\mu+\Lambda_q(a,b))^{-1}f$, $f \in L^q$, is in $W^{1,q} \cap W^{1,\frac{qd}{d-2}}$, and
\begin{equation}
\tag{$\star\star$}
\label{reg_double_star}
\begin{array}{c}
\|\nabla u \|_{q} \leq  K_1(\mu-\mu_0)^{-\frac{1}{2}} \|f\|_q, \\[2mm]
\|\nabla u \|_{\frac{qd}{d-2}} \leq  K_2 (\mu-\mu_0)^{\frac{1}{q}-\frac{1}{2}} \|f\|_q.
\end{array}
\end{equation}

\smallskip

{\rm(\textit{iii})} Let $d\geq 3$.  Assume that $\delta < 1 \vee (1+c)^{-2}$ and 
$$
c>0, \quad 
\quad 1-\frac{4c}{(d-2)^2}-c\sqrt{\delta}\biggl(2\frac{d+3}{d-2}+1 \biggr)-\delta>0,
$$
or 
$$
-1<c<0, 
\quad \quad 1-|c|-|c|\frac{4(d-1)}{(d-2)^2}-|c|\sqrt{\delta}\biggl(2\frac{d+3}{d-2}+1 \biggr)-\delta>0.
$$
Then $u \in W^{2,2}$ and (\ref{reg_double_star}) holds with $q=2$.
\end{theorem}

\begin{corollary}
\label{cor2}
 Let $d \geq 3$, $a(x)=I+c|x|^{-2}x \otimes x$, $c>-1$, $b \in \mathbf{F}_\delta$.
If
\begin{align*}
\left\{
\begin{array}{ll}
-\frac{1}{2+\frac{2}{d-3}}<c<2(d-1)(d-3), & \text{ $d \geq 4$}, \\
-\frac{1}{9}<c<\frac{1}{4}, & \text{ $d=3$},
\end{array}
\right. \qquad \text{ and $\delta>0$ is sufficiently small}
\end{align*}
or 
\begin{equation*}
\text{$|c|$ is sufficiently small and } \delta<1 \wedge \frac{4}{(d-2)^2},
\end{equation*}
then, for all $d\geq4$ and $q> d-2$ sufficiently close to $d-2$,
\[
(\mu+\Lambda_q(a,b))^{-1}L^q \subset C^{0,\gamma}, \gamma=1-\frac{d-2}{q};
\]
and, for $d= 3$,
$$
(\mu+\Lambda_2(a,b))^{-1}L^2 \subset C^{0,\gamma} \quad \gamma = \frac{1}{2}. 
$$
\end{corollary}

\begin{remark*}
In Theorem \ref{thm2}, if $\delta=0$, then the assumptions on $q$ and $c$ coincide with the ones in Theorem \ref{thm1}.
On the other hand, if $c=0$, then the assumptions on $\delta$ are reduced to $\delta<1 \wedge \frac{4}{(d-2)^2}$, so we recover the result in \cite[Lemma 5]{KS}, \cite[Theorem 3.7]{KiS}.
\end{remark*}

\textbf{2.~}Next, we consider the non-divergence form operator.

\begin{theorem}[$-a \cdot \nabla^2$]
\label{thm3}
Let $d \geq 3$, $a(x)=I+c|x|^{-2}x \otimes x$, $c>-1$.

{\rm(\textit{i})} $-a \cdot \nabla^2$ has an operator realization  $\Lambda_q(a,\nabla a)$ in $L^q$, $q \in \big[(1-\frac{d-1}{d-2}\frac{c}{1+c})^{-1}, \infty\big[$ if $0<c<d-2$, and $q \in ]1,\infty[$ if $-1<c<0$, as the (minus) generator of a  positivity preserving $L^\infty$ contraction $C_0$ semigroup.

Set $u:=(\mu+\Lambda_q(a,\nabla a))^{-1}f$, $\mu>0$, $f\in L^q$.

{\rm(\textit{ii})} Let $d \geq 4$. Assume that 
$q > d-2$ and $$-\biggl(1+\frac{1}{4}\frac{q}{d-2}\frac{(q-2)^2}{(q-1)(q+d-3)}\biggr)^{-1}<c<\frac{d-3}{2} \wedge \frac{d-2}{q-d+2}.$$

Then $u \in W^{1,q} \cap W^{1,\frac{qd}{d-2}}$, and there exist constants $K_l=K_l(d,q,c)$, $l=1,2$, such that \eqref{reg_est000} holds (for $u=(\mu+\Lambda_q(a,\nabla a))^{-1}f$).

\smallskip

{\rm(\textit{iii})} Let $d \geq 3$ and $q=2$. Assume that $-1<c<\frac{(d-2)^2}{2(2+(d-2)(d-3))}$ . Then $u \in W^{2,2}$.
\end{theorem}

\begin{corollary}
(i) Let $d\geq 4.$  For all $c\in ]0,\frac{d-3}{2}[$ and $q \in ]d-2, d + \frac{2}{d-3}[$, or 
for all $c \in]-\frac{1}{1+\frac{1}{4}\frac{(d-4)^2}{(d-3)(2d-5)}},0[$ and $q>d-2$ sufficiently close to $d-2$, 
$$(\mu+\Lambda_q(a,\nabla a))^{-1}L^q \subset C^{0,\gamma}, \quad \gamma=1-\frac{d-2}{q}.$$
(ii) For $d=3$ and all $c\in ]-1,\frac{1}{3}[$,
\[
(\mu+\Lambda_2(a,\nabla a))^{-1}L^2 \subset C^{0,\gamma}, \quad \gamma=\frac{1}{2}.
\]
\end{corollary}

\begin{remark*}
Set
$
a^\varepsilon:=I+c|x|^{-2}_\varepsilon x \otimes x$, $|x|_\varepsilon:=\sqrt{|x|^2+\varepsilon}$,  $\varepsilon>0$.
Let $d \geq 4$. Then, in the assumptions of Theorem \ref{thm3}, we have
\begin{equation}
\label{good_sol_conv}
(\mu+\Lambda_q(a,\nabla a))^{-1}=s{\mbox-}L^q{\mbox-}\lim_{\varepsilon \downarrow 0} (\mu + \Lambda_q(a^\varepsilon, \nabla a^{\varepsilon}))^{-1}.
\end{equation}
See Theorem \ref{conv_appendix} for details.
In particular, $u:=(\mu+\Lambda_q(a,\nabla a))^{-1}f$, $f \in L^q$, is a good solution of $(\mu-a\cdot\nabla^2) u = f$ in the sense of \cite{CEF}.
\end{remark*}

\begin{theorem}[$-a \cdot \nabla^2 + b \cdot \nabla$]
\label{thm4} 
Let $d \geq 3$, $a(x)=I+c|x|^{-2}x \otimes x$, $-1<c<d-2$, $b \in \mathbf{F}_\delta$.

 Set $\delta_1:=[1 \vee (1+c)^{-2}]\,\delta$, and
$$
\sqrt{\delta_2}:=\left\{
\begin{array}{ll}
\sqrt{\delta_1}+2\frac{d-1}{d-2}\frac{c}{1+c}, & 0<c<d-2, \\
\sqrt{\delta_1}, & -1<c<0.
\end{array}
\right.
$$

{\rm(\textit{i})} Assume that $\delta_2<4$. Then $-a \cdot \nabla^2 + b \cdot \nabla$ has an operator realization  $\Lambda_q(a,\nabla a + b)$ in $L^q$, $q \in \big[\frac{2}{2-\sqrt{\delta_2}}, \infty\big[$, as the (minus) generator of a  positivity preserving $L^\infty$ contraction  $C_0$ semigroup.

{\rm(\textit{ii})} Let $d \geq 4$. Assume that $q>d-2$, $\delta<1 \wedge \frac{4}{(d-2)^2}$ and $c$ satisfy one of the following two conditions:

{\rm 1)} $c>0$ and $1+c\big(1-\frac{q}{d-2}-\frac{q\sqrt{\delta}}{4}\big) \geq 0$,
and $$(q-1)\frac{(d-2)^2}{q^2} - \LL^{\rm nd}_1(c,\delta)>0,$$
\begin{align*}
\LL^{\rm nd}_1(c,\delta)\equiv \LL^{\rm nd}_1(c,\delta,q,d)&:=c(1+\theta)\frac{(d-2)^2}{q^2} + \frac{c\sqrt{\delta}}{2}\biggl[\frac{(d-2)^2}{q}+(d+3)(d-2) \biggr] \\
&+ \biggl[\frac{q^2\delta}{4}+(q-2)\frac{q\sqrt{\delta}}{2}\biggr]\frac{(d-2)^2}{q^2}, \qquad \theta=\frac{q}{d-2}.
\end{align*}

\smallskip

{\rm 2)} $-1<c<0$ and $1+c\big(1+\frac{q\sqrt{\delta}}{4}\big) \geq 0$,
and $$(q-1)\frac{(d-2)^2}{q^2} -\LL^{\rm nd}_2(-c,\delta)>0,$$
\begin{align*}
\LL_2^{\rm nd}(c,\delta) \equiv \LL_2^{\rm nd}(c,\delta,q,d)&:= c\biggl[1+(q-2)(1+\theta)\biggr]\frac{(d-2)^2}{q^2} + \frac{c\sqrt{\delta}}{2}\biggl[\frac{(d-2)^2}{q}+(d+3)(d-2) \biggr] \\
& + \biggl[\frac{q^2\delta}{4}+(q-2)\frac{q\sqrt{\delta}}{2}\biggr]\frac{(d-2)^2}{q^2},
\qquad \theta:=\frac{1}{4}\frac{q}{d-2}\frac{q-2}{q+d-3}.
\end{align*}
Then there exist constants $\mu_0=\mu_0(d,q,c,\delta)>0$ and $K_l=K_l(d,q,c,\delta)$, $l=1,2$, such that $u:=(\mu+\Lambda_q(a,\nabla a + b))^{-1}f$, $f \in L^q$, is in $W^{1,q} \cap W^{1,\frac{qd}{d-2}}$ for all $\mu>\mu_0$, and \eqref{reg_double_star} hold.

\smallskip

\rm(\textit{iii}) Let $d\geq 3$ and $q=2$. Assume that
$$
c>0, \quad \sqrt{\delta} + 2\frac{d-1}{d-2}\frac{c}{1+c}<1, \quad 1-\frac{4c}{(d-2)^2}\bigg(1+\frac{(d-2)(d-3)}{2}\bigg) - c\sqrt{\delta}\biggl(2\frac{d+3}{d-2}+1 \biggr)-\delta>0
$$
or
$$
-1<c<0, \quad \delta<(1+|c|)^{-2}, \quad \quad 1-|c|-|c|\sqrt{\delta}\biggl(2\frac{d+3}{d-2}+1 \biggr)-\delta>0.
$$
Then $u\in W^{2,2}$.
\end{theorem}

\begin{corollary}
\label{cor4}
Let $d \geq 3$, $a(x)=I+c|x|^{-2}x \otimes x$, $b \in \mathbf{F}_\delta$. Assume that 
\begin{align*}
\left\{
\begin{array}{ll}
-\frac{1}{1+\frac{1}{4}\frac{(d-4)^2}{(d-3)(2d-5)}}<c <\frac{d-3}{2}, &\text{ $d \geq 4$}, \\ 
-1<c<\frac{1}{3}, &  \text{ $d=3$}. 
\end{array}
\right. \qquad \text{and $\delta>0$ is sufficiently small,}  
\end{align*}
or 
\begin{equation*}
\text{$|c|$ is sufficiently small and } \delta<1 \wedge \frac{4}{(d-2)^2}.
\end{equation*}
Let $d \geq 4$. Then, for all $q \in ]d-2, d + \frac{2}{d-3}[$ in case of positive $c$,
and for a $q>d-2$ sufficiently close to $d-2$ in case of negative $c$, we have
$$(\mu+\Lambda_q(a,\nabla a + b))^{-1}L^q \subset C^{0,\gamma}, \quad \gamma=1-\frac{d-2}{q}$$
Let $d=3$. Then
\[
(\mu+\Lambda_2(a,\nabla a + b))^{-1}L^2 \subset C^{0,\gamma}, \;\; \gamma=\frac{1}{2}.
\]
\end{corollary}

\section{Proof of Theorem \ref{thm1}} 

In what follows, we use notation
$$
\langle h\rangle:=\int_{\mathbb R^d} h(x)dx, \quad \langle h,g\rangle:=\langle h\bar{g}\rangle.
$$

Define $t[u,v]:=\langle \nabla u \cdot a \cdot \nabla \bar{v} \rangle$, $D(t)=W^{1,2}$.
There is a unique self-adjoint operator $A\equiv A_2 \geq 0$ on $L^2$ associated with the form $t$: $D(A) \subset D(t)$, $\langle Au,v\rangle=t[u,v]$, $u \in D(A)$, $v \in D(t)$. $-A$ is the generator of a  positivity preserving $L^\infty$ contraction $C_0$ semigroup $T^t_2 \equiv e^{-tA}$,  $t \geq 0$, on $L^2$. 

By interpolation, $T^t_r:=\big[T^t_2\upharpoonright_{L^r \cap L^2}\big]_{L^r \rightarrow L^r}^{{\rm \clos}}$ determines a $C_0$ semigroup on $L^r$ for all $r \in [2,\infty[$ and hence, by self-adjointness, for all $r \in ]1,\infty[$. The (minus) generator $A_r$ of $T^t_r \,(\equiv e^{-tA_r})$ is the desired operator realization of $\nabla \cdot a \cdot \nabla$ on $L^r$, $r \in ]1,\infty[$. One can furthermore show that $T^t_1:=\big[T^t_2\upharpoonright_{L^1 \cap L^2}\big]_{L^1 \rightarrow L^1}^{{\rm \clos}}$ is a $C_0$ semigroup. This completes the proof of the  assertion (\textit{i}) of the theorem.

To prove (\textit{ii}), we will need the following notation and auxiliary results.  Define the smoothed out matrices $a^\varepsilon=(a^\varepsilon_{ij})$, $1 \leq i,j \leq d$, $\varepsilon>0$ by
$$
a^\varepsilon_{ij}:=\delta_{ij}+c|x|^{-2}_\varepsilon x_ix_j, \quad |x|_\varepsilon:=\sqrt{|x|^2+\varepsilon}.
$$
Set $u \equiv u^\varepsilon = (\mu + A_q^\varepsilon)^{-1} f$, $A_q^\varepsilon:=A_q(a^\varepsilon)$, $0 \leq f \in C_c^1$. Clearly $a^{\varepsilon} \in C^\infty$ and $ 0 \leq u^\varepsilon \in W^{3,q}$. 
Denote $w \equiv w^\varepsilon:=\nabla u^\varepsilon$, 
$$
I_q:=\sum_{r=1}^d\langle (\nabla_r w)^2 |w|^{q-2} \rangle, \quad
J_q:=\langle (\nabla |w|)^2 |w|^{q-2}\rangle,$$
$$
\bar{I}_{q,\chi}:= \langle  \bigl( x \cdot \nabla w \bigr)^2 \chi|x|^{-2}|w|^{q-2}\rangle, \quad
\bar{J}_{q,\chi}:= \langle (x \cdot \nabla |w|)^2 \chi|x|^{-2}  |w|^{q-2}\rangle, \qquad \chi:=|x|^2|x|_\varepsilon^{-2},
$$
$$H_{q,\chi}:=\langle \chi|x|^{-2}|w|^q \rangle, \quad H_{q,\chi^2}:=\langle \chi^2|x|^{-2}|w|^q \rangle, \quad G_{q,\chi^2}:=\langle \chi^2|x|^{-4} (x \cdot w)^2
|w|^{q-2}\rangle.
$$

\textbf{1.~}The following inequality plays an important role in the proof of Theorem \ref{thm1}. 

\begin{lemma}[Hardy-type inequality]
\[
\label{hardy_type_ineq}
\frac{d^2}{4}H_{q,\chi}-(d+2)H_{q,\chi^2} + 3H_{q,\chi^3} \leq \frac{q^2}{4} \bar{J}_{q,\chi}  \tag{$\mathbf{HI}$}
\]
\end{lemma}
\begin{proof}
Set $F:=|x|_\varepsilon^{-1}|w|^\frac{q}{2}$.
Then
$$
\frac{q^2}{4}\bar{J}_{q,\chi} = \bigl\langle \bigl(|x|_\varepsilon^{-1}x \cdot \nabla |w|^{\frac{q}{2}}\bigr)^2\bigr\rangle = \bigl\langle (x \cdot \nabla F + \chi F)^2\bigr\rangle = \langle (x \cdot \nabla F)^2 \rangle + \langle \chi^2 F^2 \rangle + 2\langle x \cdot \nabla F,\chi F\rangle.
$$
\eqref{hardy_type_ineq} follows from the inequality $\langle (x \cdot \nabla F)^2 \rangle \equiv\|x\cdot \nabla F \|_2^2 \geq \frac{d^2}{4}\|F\|_2^2 \equiv \frac{d^2}{4} H_{q,\chi}$ and the equalities
 $$2 \langle x \cdot \nabla F,\chi F\rangle=-d\langle \chi F^2\rangle - \langle F^2, x \cdot \nabla \chi\rangle,  \quad x \cdot \nabla \chi = 2 \bigl(\frac{|x|^2}{|x|_\varepsilon^2} - \frac{|x|^4}{|x|_\varepsilon^4}\bigr)=2\chi (1-\chi).$$
\end{proof}

The following equalities are crucial steps in the proof of Theorem \ref{thm1}. 

\begin{lemma}[The basic equalities]
\label{basic_ineq_lemma}
\begin{align}
&\mu \langle |w|^q \rangle + I_q + c\bar{I}_{q,\chi}+ (q-2)(J_q + c\bar{J}_{q,\chi}) 
 - c\biggl(1+(q-2)\frac{d}{q}\biggr)H_{q,\chi} + 2c(d-1)G_{q,\chi^2} \notag \\
&+2c\frac{q-2}{q} H_{q,\chi^2} + 8c\varepsilon \langle |x|_\varepsilon^{-6} (x \cdot w)^2 |w|^{q-2}\rangle =  \beta_1 + \langle f, \phi \rangle
\tag{${\rm BE}_+$},
\label{be_+}
\end{align}
\begin{align}
&\mu \langle |w|^q \rangle + I_q^{} + c\bar{I}_{q,\chi}^{} + (q-2)(J_q^{} + c\bar{J}_{q,\chi}^{})  - c\biggl(1+(q-2)\frac{d}{q}\biggr)H_{q,\chi}^{} + cd G_{q,\chi^2}^{} \notag \\
&+2c\frac{q-2}{q}H_{q,\chi^2} + 4c\varepsilon \langle |x|_\varepsilon^{-6} (x \cdot w)^2 |w|^{q-2}\rangle = - \frac{1}{2}\beta_2^{}+ \langle f, \phi^{} \rangle,
\tag{${\rm BE}_-$}
\label{be_-}
\end{align}
where $\phi^{}=- \nabla \cdot (w |w|^{q-2})$,
$$
\beta_1^{}:=-2c\langle |x|_\varepsilon^{-4} x \cdot w, x \cdot (x \cdot \nabla w)|w|^{q-2}\rangle, \qquad \beta_2^{}:=-2c(q-2)\langle |x|_\varepsilon^{-4} (x \cdot w)^2 x \cdot \nabla |w|,|w|^{q-3} \rangle.
$$
\end{lemma}

\begin{remark*}
Below we use the representation \eqref{be_+} in case $c>0$, and the representation \eqref{be_-} in case $c<0$. (One could still use \eqref{be_+} for $c<0$ or \eqref{be_-} for $c>0$, but this would lead to more restrictive constraints on $c$.)
\end{remark*}

\begin{proof}[Proof of Lemma \ref{basic_ineq_lemma}]
Set $[F,G]_-:=FG-GF$.
We multiply $\mu u^{} + A_q^\varepsilon u^{}= f $ by $\phi^{}$ and integrate:
\begin{equation*}
\mu \langle |w|^q \rangle +\langle A_q^\varepsilon w^{}, w|w|^{q-2} \rangle + \langle [\nabla,A_q^\varepsilon]_-u^{}, w|w|^{q-2}\rangle = \langle f, \phi^{} \rangle,
\end{equation*}
\begin{equation}
\label{be_with_commutator}
\mu \langle |w|^q \rangle + I_q^{} + c\bar{I}_{q,\chi}^{} + (q-2)(J_q + c\bar{J}_{q,\chi}^{})+ \langle [\nabla,A_q^\varepsilon]_-u^{}, w|w|^{q-2}\rangle = \langle  f, \phi^{} \rangle.
\end{equation}
The term to evaluate: $\langle [\nabla,A_q^\varepsilon]_- u^{},w|w|^{q-2}\rangle \equiv \langle [\nabla_r,A_q^\varepsilon]_- u^{},w_r|w|^{q-2}\rangle:=\sum_{r=1}^d\langle [\nabla_r,A_q^\varepsilon]_- u^{},w_r|w|^{q-2}\rangle$.
Note that 
$$[\nabla_r,A_q^\varepsilon]_-=-\nabla \cdot (\nabla_ra^\varepsilon) \cdot \nabla, \qquad (\nabla_r a^\varepsilon)_{ik} = c|x|_\varepsilon^{-2}\delta_{ri}x_k + c(|x|_\varepsilon^{-2}\delta_{rk}x_i - 2|x|_\varepsilon^{-4}x_i x_k x_r),$$
\begin{align*}
&\langle [\nabla_r,A_q^\varepsilon]_- u^{},w_r|w|^{q-2}\rangle \\
&= -c\langle w_k\nabla_i(|x|_\varepsilon^{-2}\delta_{ri}x_k)+|x|_\varepsilon^{-2}\delta_{ri}x_k\nabla_i w_k,w_r|w|^{q-2}\rangle
+c\langle(|x|_\varepsilon^{-2}\delta_{rk}x_i-2|x|_\varepsilon^{-4}x_ix_kx_r)w_k,\nabla_i(w_r|w|^{q-2})\rangle\\
&=:\alpha_1^{} + \alpha_2^{},
\end{align*}
\begin{align*}
&\alpha_1^{}=-c\langle (|x|_\varepsilon^{-2}\delta_{rk}-2|x|_\varepsilon^{-4}\delta_{ri}x_kx_r)w_k + |x|_\varepsilon^{-2}x \cdot \nabla w_r,w_r|w|^{q-2}\rangle \\
&=-c \langle |x|_\varepsilon^{-2}|w|^q\rangle + 2c\langle |x|_\varepsilon^{-4} (x \cdot w)^2 |w|^{q-2}\rangle - c\langle |x|_\varepsilon^{-2}x\cdot\nabla|w|,|w|^{q-1} \rangle.
\end{align*}
Then
$$
\alpha_1^{}=-c\biggl(1-\frac{d-2}{q}\biggr)H_{q,\chi}^{} + 2cG_{q,\chi^2}^{} + 2\frac{c}{q}\varepsilon \langle |x|_\varepsilon^{-4}|w|^q\rangle
$$
due to $\langle |x|_\varepsilon^{-2}x\cdot\nabla|w|,|w|^{q-1} \rangle=\frac{1}{q}\langle |x|_\varepsilon^{-2}x \cdot \nabla |w|^q\rangle=-\frac{1}{q}\langle |w|^q \nabla \cdot (x|x|_\varepsilon^{-2})\rangle=-\frac{d}{q}H_{q,\chi}^{} + \frac{2}{q}\langle|x|^2|x|_\varepsilon^{-4}|w|^{q}\rangle=-\frac{d-2}{q}H_{q,\chi}^{} -\frac{2}{q}\varepsilon \langle |x|_\varepsilon^{-4}|w|^q\rangle$, and
$$
\alpha_2^{}=c \langle |x|_\varepsilon^{-2}w,x \cdot \nabla (w|w|^{q-2}) \rangle - 2c\langle |x|_\varepsilon^{-4} x \cdot w, x \cdot (x \cdot \nabla(w|w|^{q-2})) \rangle.
$$
Then
\begin{align*}
\alpha_2^{}& =\beta_1^{} + \beta_2^{} + c \langle |x|_\varepsilon^{-2} x \cdot \nabla|w|,|w|^{q-1}\rangle + c(q-2)\langle |x|_\varepsilon^{-2} x \cdot \nabla |w|, |w|^{q-1}\rangle \\
&=\beta_1^{} + \beta_2^{} + c(q-1)\biggl(\frac{d-2}{q}H_{q,\chi}^{} + \frac{2}{q}\varepsilon \langle |x|_\varepsilon^{-4},|w|^{q} \rangle \biggr).
\end{align*}
In view of 
$$
\beta_1^{}=-\frac{1}{2}\beta_2^{} + c(d-2)G_{q,\chi^2}^{} + 4c\varepsilon\langle |x|_\varepsilon^{-6} (x \cdot w)^2 |w|^{q-2} \rangle,
$$
we rewrite $\alpha_1^{} + \alpha_2^{} = \langle [\nabla, A_q^\varepsilon]_- u^{},w|w|^{q-2}\rangle$ in two ways:
\begin{align*}
&\langle [\nabla,A_q^\varepsilon]_- u^{}, w|w|^{q-2}\rangle = -\beta_1^{} - c\biggl(1+(q-2)\frac{d-2}{q}\biggr)H_{q,\chi}^{} + 2c(d-1)G_{q,\chi^2}^{} \notag \\
&-2c\frac{q-2}{q}\varepsilon \langle |x|_\varepsilon^{-4}|w|^q\rangle + 8c\varepsilon \langle |x|_\varepsilon^{-6} (x \cdot w)^2 |w|^{q-2}\rangle 
\end{align*}
and
\begin{align*}
&\langle [\nabla,A_q^\varepsilon]_- u^{}, w|w|^{q-2}\rangle = \frac{1}{2}\beta_2 - c\biggl(1+(q-2)\frac{d-2}{q}\biggr)H_{q,\chi}^{} + cd G_{q,\chi^2}^{} \notag \\
&-2c\frac{q-2}{q}\varepsilon \langle |x|_\varepsilon^{-4}|w|^q\rangle + 4c\varepsilon \langle |x|_\varepsilon^{-6} (x \cdot w)^2 |w|^{q-2}\rangle. 
\end{align*}
The last two identities applied in \eqref{be_with_commutator} yield \eqref{be_+}, \eqref{be_-}.
\end{proof}

\smallskip

\textbf{2.~}Next, we estimate from above the term $\langle f, \phi \rangle$ in the right-hand side of \eqref{be_+}, \eqref{be_-}.
\begin{lemma} 
\label{f_est_lem}
For each $\varepsilon_0>0$ there exists a constant $C(\varepsilon_0)<\infty$  such that 
$$
\langle f, \phi \rangle \leq \varepsilon_0 (I_{q}^{} +  J^{}_{q} + H_q^{}) + C(\varepsilon_0) \|w\|_q^{q-2} \|f\|^2_q,
$$
where $H_q:=\langle |x|^{-2}|w|^q\rangle.$
\end{lemma}

\begin{proof}[Proof of Lemma \ref{f_est_lem}]
Clearly,
\begin{align*}
 \langle f, \phi^{} \rangle & =\langle f, (- \Delta u)|w|^{q-2}\rangle - (q-2) \langle f, |w|^{q-3} w \cdot \nabla |w|\rangle=:F_1+F_2.
\end{align*}
Since $-\Delta u=\nabla \cdot (a^\varepsilon-1)\cdot w -\mu u  +f$
and 
\begin{align*}
F_1& =\langle \nabla \cdot (a^\varepsilon-1)\cdot w, |w|^{q-2}f\rangle  + \langle (-\mu u  +f),|w|^{q-2}f \rangle \\
& (\text{we expand the first term using $\nabla a^\varepsilon=c(d+1)x|x|_\varepsilon^{-2}-2c |x|^2|x|_\varepsilon^{-4}x$}) \\
& = c(d+1)\langle|x|_\varepsilon^{-2}x\cdot w, |w|^{q-2}f\rangle - 2c\langle 
\chi|x|_\varepsilon^{-2}x\cdot w, |w|^{q-2}(-b\cdot w + f) \rangle \\
& + c \langle |x|_\varepsilon^{-2} x \cdot (x \cdot \nabla w), |w|^{q-2}f \rangle 
 + \langle (-\mu u +f),|w|^{q-2}f \rangle.
\end{align*}

We bound from above $F_1$ and $F_2$  by applying consecutively the following estimates:

\smallskip

1) $\langle|x|_\varepsilon^{-2}x\cdot w, |w|^{q-2}f \rangle \leq H_{q}^{\frac{1}{2}} \|w\|_q^\frac{q-2}{2} \|f\|_q$.

2) $\langle\chi |x|_\varepsilon^{-2}x\cdot w, |w|^{q-2}f \rangle \leq H_{q}^{\frac{1}{2}} \|w\|_q^\frac{q-2}{2} \|f\|_q$.

3) $\langle |x|_\varepsilon^{-2} x \cdot (x \cdot \nabla w), |w|^{q-2} f \rangle \leq (\bar{I}^{}_{q,\chi})^{\frac{1}{2}}\|w\|_q^\frac{q-2}{2} \|f\|_q.$

\smallskip

4) $\langle -f, |w|^{q-2} \mu u \rangle \leq 0.$

\smallskip

5) $ \langle f, |w|^{q-2} f \rangle \leq \|w\|_q^{q-2} \|f\|_q^2 .$

6) $ (q-2) \langle -f, |w|^{q-3} w \cdot \nabla |w| \rangle \leq (q-2) J_{q}^\frac{1}{2}\|w\|_q^\frac{q-2}{2} \|f\|_q .$

\smallskip

1)-6) and the standard quadratic estimates now yield the lemma.
\end{proof}

We choose $\varepsilon_0>0$ in Lemma \ref{f_est_lem} so small that in the estimates below we can ignore 
$\varepsilon_0 (I_{q}^{} +  J^{}_{q} + H_q^{})$.

\medskip

\textbf{3.}~We will use \eqref{be_+}, \eqref{be_-} and Lemma \ref{f_est_lem} to prove the following inequality
\begin{equation}
\label{req_est4}
\mu \langle |w|^q \rangle + \eta J_q^{} \leq C\|w\|_q^{q-2} \|f\|^2_q, \qquad C=C(\varepsilon_0)
\end{equation}
for some $\eta=\eta(q,d,\varepsilon_0)>0$.

\medskip

\textbf{\textit{Case $c > 0$}.\;}In \eqref{be_+} we omit the term $8c\varepsilon\langle |x|_\varepsilon^{-6} (x \cdot w)^2 |w|^{q-2}\rangle$, %(which will not affect the resulting assumptions on $c$),
obtaining
\begin{align*}
&\mu \langle |w|^q \rangle + I_q + c\bar{I}_{q,\chi}+ (q-2)(J_q + c\bar{J}_{q,\chi}) 
 - c\biggl(1+\frac{q-2}{q}d\biggr)H_{q,\chi} + 2c(d-1)G_{q,\chi^2} \notag \\
&+2c\frac{q-2}{q} H_{q,\chi^2} \leq \beta_1 + \langle f, \phi \rangle.
\end{align*}
Estimating $\beta_1$ from above using the standard quadratic estimates,
$$\beta_1^{} \leq 2c\langle |x|_\varepsilon^{-4} (x\cdot w)^2|w|^{q-2}\rangle^{\frac{1}{2}} \langle |x|_\varepsilon^{-4} |x|^2 (x \cdot \nabla w)^2 |w|^{q-2}\rangle^{\frac{1}{2}} \leq 2c (G_{q,\chi^2}^{} \bar{I}_{q,\chi}^{})^{\frac{1}{2}} \leq c\theta \bar{I}_{q,\chi}^{} + c\theta^{-1} G_{q,\chi^2}^{}$$ 
($\theta>0$), and then applying Lemma \ref{f_est_lem}, we have
\begin{align*}
&\mu \langle |w|^q \rangle + I_q + c(1-\theta)\bar{I}_{q,\chi}+ (q-2)(J_q + c\bar{J}_{q,\chi}) 
 - c\biggl(1+\frac{q-2}{q}d\biggr)H_{q,\chi} + c\big(2(d-1)- \theta^{-1} \big)G_{q,\chi^2} \notag \\
&+2c\frac{q-2}{q} H_{q,\chi^2}\leq C\|w\|_q^{q-2} \|f\|^2_q.
\end{align*}
Let $0<\theta <1$. Using the inequalities
$J_q^{} \leq I_q$, $\bar{J}_{q,\chi} \leq \bar{I}_{q,\chi}$ and $\frac{4}{q^2}\left(\frac{d^2}{4}H_{q,\chi}^{}-(d+2)H_{q,\chi^2}^{} + 3H_{q,\chi^3}^{} \right) \leq \bar{J}_{q,\chi}$, see \eqref{hardy_type_ineq},
we have
\begin{align*}
\mu \langle |w|^q \rangle 
 + \eta J_q^{}  + (-\eta+q-1)J_q +  \bigl[2c(d-1)- c\theta^{-1} \bigr]G_{q,\chi^2} + c \langle M(\chi)|x|^{-2}|w|^{q}\rangle   \leq C\|w\|_q^{q-2} \|f\|^2_q,
\end{align*}
where
$$
M(\chi):=\bigg[\bigl(q-1-\theta\bigr)\frac{4}{q^2}\biggl(\frac{d^2}{4}-(d+2)
\chi+3\chi^2\biggr) - \biggl(1+\frac{q-2}{q}d\biggr) + 2\frac{q-2}{q} \chi\bigg]\chi,
$$
i.e.
\[
M(\chi):=[\mathfrak a\chi^2 + \mathfrak b\chi+\mathfrak c_0]\chi,
\]
where
\[
\mathfrak a=\frac{12}{q^2}(q-1-\theta), \;\; \mathfrak b=-4(q-1-\theta)\frac{d+2}{q^2} + 2\frac{q-2}{q}, \;\; \mathfrak c_0=\frac{d^2}{q^2}(q-1-\theta)+\frac{2d}{q}-1-d .
\]
Elementary arguments show that the choice $\theta:=\frac{1}{2(d-1)}$ is the best possible. In particular,
\[
\min_{0\leq t \leq 1} M(t)= M(1)<0.
\] 
Since $-\eta+q-1>0$ for all $\eta>0$ sufficiently small, we can use $J_q \geq \frac{(d-2)^2}{q^2}H_q$, obtaining
\begin{align}
\mu\langle |w|^q \rangle  + \eta J_q^{}  + \biggl((-\eta+q-1)\frac{(d-2)^2}{q^2} +c M(1)\biggr)H_q \leq C\|w\|_q^{q-2} \|f\|^2_q, \notag 
\end{align}
Recalling the assumption $(q-1)\frac{(d-2)^2}{q^2} - c \ell_1>0$, $\ell_1= - M(1)$, it is seen that there exists $\eta >0$ such that $(-\eta+q-1)\frac{(d-2)^2}{q^2} \geq c \ell_1$. \eqref{req_est4} is proved for $0<c<\frac{(q-1)(d-2)^2}{q^2 \ell_1}.$

The choice of $\theta \in [1,1+c^{-1}]$ leads to sub-optimal constraints on $c$ and $q$.

\smallskip

\textit{\textbf{Case $-1<c<0$.}}~Set $s:=|c|$. In \eqref{be_-}, we estimate ($\theta>0$)
$$
|\beta_2^{}| \leq 2s(q-2)\bigl(\theta \bar{J}_{q,\chi}^{}  + 4^{-1}\theta^{-1} G_{q,\chi^2}^{}\bigr).
$$
By \eqref{be_-} and Lemma \ref{f_est_lem},
\begin{align*}
&\mu \langle |w|^q\rangle + I_q^{} - s\bar{I}_{q,\chi}^{} + (q-2)(J_q^{} - s(1+\theta) \bar{J}_{q,\chi}^{}) + s\left(1+(q-2)\frac{d}{q} \right)H_{q,\chi}^{}  \\
&-2s\frac{q-2}{q}H_{q,\chi^2}^{} - sd G_{q,\chi^2}^{} -4s G_{q,\chi^2}^{}  + 4sG_{q,\chi^3}^{}  - s(q-2)\frac{1}{4\theta}G_{q,\chi^2}^{}\leq   C\|w\|_q^{q-2} \|f\|^2_q. \notag
\end{align*}
Clearly, $I_q^{} - s\bar{I}_{q,\chi}^{} + (q-2)(J_q^{} - s(1+\theta) \bar{J}_{q,\chi}^{}) \geq (q-1-s-s(q-2)(1+\theta))J_q $.
 Therefore
\begin{align*}
&\mu \langle |w|^q\rangle + \bigl(q-1-s - s(q-2)(1+\theta)\bigr)J_q^{}
+s\left(1+(q-2)\frac{d}{q} \right)H_{q,\chi}^{} \\
&-2s\frac{q-2}{q}H_{q,\chi^2}^{} - sd G_{q,\chi^2}^{} -4s G_{q,\chi^2}^{}  + 4sG_{q,\chi^3}^{} - s(q-2)\frac{1}{4\theta}G_{q,\chi^2}^{}\leq  C\|w\|_q^{q-2} \|f\|^2_q.
\end{align*}
Using $H_{q,\chi} \geq H_{q,\chi^2}$ and $G_{q,\chi} \leq H_{q,\chi}$, we obtain
\begin{align*}
&\mu \langle |w|^q\rangle + \bigl(q-1-s - s(q-2)(1+\theta)\bigr)J_q^{}
+s\left(1+(q-2)\frac{d-2}{q} \right)G_{q,\chi}^{} \\
& - sd G_{q,\chi^2}^{} -4s G_{q,\chi^2}^{}  + 4sG_{q,\chi^3}^{} - s(q-2)\frac{1}{4\theta}G_{q,\chi^2}^{}\leq  C\|w\|_q^{q-2} \|f\|^2_q,
\end{align*}
i.e.
\begin{align*}
&\mu \langle |w|^q\rangle + \eta J_q+ \bigl(-\eta+q-1-s - s(q-2)(1+\theta)\bigr)J_q^{}
+s\langle M(\chi) |x|^{-4}(x \cdot w)^2|w|^{q-2}\rangle \leq  C\|w\|_q^{q-2} \|f\|^2_q,
\end{align*}
where
$$
M(\chi):=\biggl[1+(q-2)\frac{d-2}{q}+\biggl(-d-4+4\chi - (q-2)\frac{1}{4\theta}\biggr)\chi\biggr]\chi,
$$
i.e.
\[
M(\chi)=[\mathfrak a\chi^2 + \mathfrak b\chi+ \mathfrak c_0]\chi,
\]
where
\[
\mathfrak a=4, \;\; \mathfrak b=-d-4-\frac{1}{2}(q-2)\frac{d-2}{q},\;\;\mathfrak c_0=1+(q-2)\frac{d-2}{q}.
\]
Select
$
\theta:=\frac{1}{2}\frac{q}{d-2}.
$ (Motivation: Below we estimate $I_q^{} - s\bar{I}_{q,\chi}^{} + (q-2)(J_q^{} - s(1+\theta) \bar{J}_{q,\chi}^{}) \geq (q-1-s-s(q-2)(1+\theta))J_q \geq (q-1-s-s(q-2)(1+\theta))\frac{(d-2)^2}{q^2}G_q$, so estimating the terms  involving $\theta$ in the resulting inequality as $ -s(q-2)\theta\frac{(d-2)^2}{q^2} G_q - (q-2)\frac{1}{4\theta}G_{q,\chi^2}^{} \geq \bigl(-s(q-2)\theta\frac{(d-2)^2}{q^2}  - (q-2)\frac{1}{4\theta}\bigr)G_{q}^{}$, we arrive clearly at $\theta=\frac{1}{2}\frac{q}{d-2}$.)

Elementary arguments show that $\min_{0\leq t\leq 1}M(t) = M(1)<0.$
By the assumptions of the theorem, 
\begin{equation*}
(-\eta+q-1-s - s(q-2)(1+\theta))\frac{(d-2)^2}{q^2} +sM(1) \geq 0.
\end{equation*}
Thus, by $J_q \geq \frac{(d-2)^2}{q^2} H_q$ and $H_q \geq G_q$,
\begin{align*}
\mu \langle |w|^q \rangle  + \eta J_q^{} + \bigl[(-\eta+q-1-s - s(q-2)(1+\theta))\frac{(d-2)^2}{q^2} +sM(1)  \bigr]G_q \leq C\|w\|_q^{q-2} \|f\|^2_q,
\end{align*}
or, setting $\ell_2:=[1+(q-2)(1+\theta)]\frac{(d-2)^2}{q^2} -M(1)$, 
\begin{align*}
\mu \langle |w|^q \rangle  +  \eta J_q^{} +\bigg[(-\eta+q-1)\frac{(d-2)^2}{q^2} +c\ell_2 \bigg] G_q    \leq C\|w\|_q^{q-2} \|f\|^2_q.
\end{align*}
\eqref{req_est4} is proved. 

\medskip

\textbf{4.~}By \eqref{req_est4}, 
$
\mu\|w\|_q^q \leq C\|w\|_q^{q-2}\|f\|_q^2$, $w=\nabla u^\varepsilon$, $\varepsilon>0,
$
and so $$ \|\nabla u^\varepsilon\|_{q} \leq  K_1\mu^{-\frac{1}{2}} \|f\|_q, \quad K_1:=C^{\frac{1}{2}}.$$
Again by \eqref{req_est4},
$\eta J_q \leq C\|w\|_q^{q-2}\|f\|_q^2$, $J_q=\frac{4}{q^2}\|\nabla |w|^\frac{q}{2}\|^2_2$, so in view of the previous inequality $\eta \|\nabla |\nabla u^\varepsilon|^\frac{q}{2}\|^2_2 \leq \frac{q^2}{4} C K^{q-2}_1\mu^{1-\frac{q}{2}} \|f\|^q_q$. 
The Sobolev Embedding Theorem now yields $$\|\nabla u^\varepsilon \|_{q j} \leq  K_2\mu^{\frac{1}{q}-\frac{1}{2}} \|f\|_q, \quad K_2:=C_{S}\eta^{-\frac{1}{q}}(q^2/4)^{\frac{1}{q}} C^{\frac{1}{q}}K_1^{\frac{q-2}{q}}.$$
Since the weak gradient in $L^q$ is closed, Theorem \ref{conv_appendix}(\textit{i}) (with $b=0$)
yields $ \|\nabla u\|_{q} \leq  K_1\mu^{-\frac{1}{2}} \|f\|_q$, $\|\nabla u \|_{q j} \leq  K_2\mu^{\frac{1}{q}-\frac{1}{2}} \|f\|_q$ for $u=(\mu+A_q)^{-1}f$, $0 \leq f \in C_c^\infty$, and thus for all $f \in L^q$. 

We have proved (\textit{ii}).

\smallskip

\textit{Proof of {\rm(}\textit{iii}{\rm)}}. Let $q=2$, $d \geq 3$. 
The arguments above yield ($I_2 \equiv I_2(u^\varepsilon)$)

(a) For $c>0$,
\begin{align*}
&\langle [\nabla,A_2^\varepsilon]_-u,w \rangle = -\beta_1 - cH_{2,\chi} + 2c(d-1)G_{2,\chi^2},\\
&\mu \|w\|_2^2 +I_2 +c\bar{I}_{2,\chi} - \beta_1 -cH_{2,\chi} +2c(d-1)G_{2,\chi^2} = \langle f, -\nabla \cdot w \rangle\\
&\beta_1=-2c\langle |x|^{-4}x\cdot w, x\cdot(x\cdot \nabla w)\rangle.
\end{align*}
By $\beta_1 \leq 2c\sqrt{G_{2,\chi^2}\bar{I}_{2,\chi}}\leq c\bar{I}_{2,\chi} + cG_{2,\chi^2}$, $G_{2,\chi^2} \leq H_{2,\chi}$ and $I_2\geq\frac{(d-2)^2}{4}H_{2,\chi}$, it follows that $1-c\frac{4}{(d-2)^2}>0 \Rightarrow I_2\leq K\|f\|_2^2;$

(b) For $c<0$,
\begin{align*}
&\langle [\nabla,A_2^\varepsilon]_-u,w \rangle = \frac{1}{2}\beta_2 - cH_{2,\chi} + cdG_{2,\chi^2}, \quad \beta_2=0,\\
&\mu \|w\|_2^2 +I_2 +c\bar{I_2} -cH_{2,\chi} +cdG_{2,\chi^2} = \langle f, -\nabla \cdot w \rangle, \quad I_2\geq\bar{I_2}, \\
&\mu \|w\|_2^2 + (1-|c|)I_2 + |c|H_{2,\chi}-|c|dG_{2,\chi^2} \leq \langle f,-\nabla \cdot w  \rangle.
\end{align*}
Thus, by $G_{2,\chi^2} \leq H_{2,\chi} \leq \frac{4}{(d-2)^2}I_2$, we conclude that $1-|c| + |c|(1-d)\frac{4}{(d-2)^2} > 0 \Rightarrow I_2(u^\varepsilon)\leq K\|f\|_2^2.$ 

By passing to the limit $\varepsilon \downarrow 0$, using Theorem \ref{conv_appendix}, we obtain $I_2(u)  \leq K\|f\|_2$. Therefore,
$u \in W^{2,2}$.

\smallskip

The proof of Theorem \ref{thm1} is completed. \qed

\section{Proof of Theorem \ref{thm2}}

\textit{Proof of (i)}. A vector field $b:\mathbb R^d \rightarrow \mathbb R^d$ belongs to $\mathbf F_{\delta_1}(A)$, $\delta_1>0$, the class of form-bounded vector fields (with respect to $A \equiv A_2$), if $b_a^2:=b \cdot a^{-1} \cdot b
 \in L^1_\loc$ and there exists a constant $\lambda=\lambda_{\delta_1}>0$ such that 
$$\|b_a(\lambda + A)^{-\frac{1}{2}}\|_{2\rightarrow 2}\leq \sqrt{\delta_1}.$$
It is easily seen that if $b \in \mathbf{F}_\delta$, then 
$b \in \mathbf{F}_{\delta_1}(A)$, where $\delta_1:=\delta$ if $c>0$, and $\delta_1:=\delta (1+c)^{-2}$  if $-1<c<0$. By our assumption, $\delta_1<4$. 
Therefore, by \cite[Theorem 3.2]{KiS}, $-\nabla \cdot a \cdot \nabla + b \cdot \nabla$ has an operator realization  $\Lambda_q(a,b)$ in $L^q$, $q \in \big[\frac{2}{2-\sqrt{\delta_1}}, \infty\big[$, as the (minus) generator of a  positivity preserving $L^\infty$ contraction quasi contraction $C_0$ semigroup. Moreover, $(\mu+\Lambda_q(a,b))^{-1}$ is well defined on $L^q$ for all $\mu>\frac{\lambda \delta}{2(q-1)}$.
This completes the proof of (\textit{i}).

\smallskip

\textit{Proof of (ii)}. Set
$
a^\varepsilon:=I+c|x|^{-2}_\varepsilon x \otimes x$, $|x|_\varepsilon:=\sqrt{|x|^2+\varepsilon}$,  $\varepsilon>0$. Put $A^\varepsilon=A(a^\varepsilon)$.
It is clear that $b \in \mathbf{F}_{\delta_1}(A^\varepsilon)$ for all $\varepsilon>0$.

Let $\mathbf 1_n$ denote the indicator of  $\{x \in \mathbb R^d \mid  \; |x| \leq n,  |b(x)| \leq n \}$, and set $b_n := \gamma_{\epsilon_n} \ast \mathbf 1_n b \in C^\infty$, where $\gamma_{\epsilon}$ is the K.\,Friedrichs mollifier, $\epsilon_n \downarrow 0$.
Since our assumptions on $\delta$ and thus $\delta_1$ involve strict inequalities only, we can select $\epsilon_n \downarrow 0$ so that  $b_n \in \mathbf{F}_{\delta_1}(A^\varepsilon)$, $\varepsilon>0$, $n \geq 1$. 
Therefore, in view of the previous discussion, $(\mu+\Lambda_q(a^\varepsilon,b_n))^{-1}$ is well defined on $L^q$, $\mu>\frac{\lambda\delta}{2(q-1)}$, $\varepsilon>0$, $n \geq 1$. 
Here $\Lambda_q(a^\varepsilon, b_n)=-\nabla \cdot a^\varepsilon \cdot \nabla + b_n \cdot \nabla$, $D(\Lambda_q(a^\varepsilon, b_n))=W^{2,q}$.

Define $0 \leq u \equiv u^{\varepsilon,n} := (\mu + \Lambda_q(a^\varepsilon,b_n))^{-1} f$, $0 \leq f \in C_c^1$. Then $u \in W^{3,q}$.
Set $w \equiv w^{\varepsilon,n}:=\nabla u^{\varepsilon,n}$ and
$$
I_q:=\langle (\nabla_r w)^2 |w|^{q-2}  \rangle, \quad
J_q:=\langle (\nabla |w|)^2 |w|^{q-2}\rangle,$$
$$
\bar{I}_{q,\chi}:= \langle  \bigl( x \cdot \nabla w \bigr)^2 \chi|x|^{-2}  |w|^{q-2} \rangle, \quad
\bar{J}_{q,\chi}:= \langle  (x \cdot \nabla |w|)^2 \chi|x|^{-2}|w|^{q-2}  \rangle, \quad \chi=|x|^2|x|_\varepsilon^{-2},
$$
$$H_{q,\chi}:=\langle \chi|x|^{-2}|w|^q \rangle, \quad H_{q,\chi^2}:=\langle \chi^2|x|^{-2}|w|^q \rangle, \quad G_{q,\chi^2}:=\langle \chi^2|x|^{-4} (x \cdot w)^2
|w|^{q-2}\rangle.
$$
Below we follow closely the proof of Theorem \ref{thm1}. 

\medskip

\textbf{1.~}We repeat the proof of Lemma \ref{basic_ineq_lemma}, where in the right-hand side of \eqref{be_+}, \eqref{be_-}
we now get an extra term $\langle -b_n\cdot w, \phi^{} \rangle$:
\begin{align}
&\mu \langle |w|^q \rangle + I_q + c\bar{I}_{q,\chi}+ (q-2)(J_q + c\bar{J}_{q,\chi}) 
 - c\biggl(1+(q-2)\frac{d}{q}\biggr)H_{q,\chi} + 2c(d-1)G_{q,\chi^2} \notag \\
&+2c\frac{q-2}{q} H_{q,\chi^2} + 8c\varepsilon \langle |x|_\varepsilon^{-6} (x \cdot w)^2 |w|^{q-2}\rangle =  \beta_1 + 
\langle -b_n\cdot w, \phi \rangle + \langle f, \phi \rangle
\tag{${\rm BE}_{+,b}$},
\label{be_+b}
\end{align}
\begin{align}
&\mu \langle |w|^q \rangle + I_q^{} + c\bar{I}_{q,\chi}^{} + (q-2)(J_q^{} + c\bar{J}_{q,\chi}^{})  - c\biggl(1+(q-2)\frac{d}{q}\biggr)H_{q,\chi}^{} + cd G_{q,\chi^2}^{} \notag \\
&+2c\frac{q-2}{q}H_{q,\chi^2} + 4c\varepsilon \langle |x|_\varepsilon^{-6} (x \cdot w)^2 |w|^{q-2}\rangle = - \frac{1}{2}\beta_2^{}+ \langle -b_n\cdot w, \phi^{} \rangle+ \langle f, \phi^{} \rangle,
\tag{${\rm BE}_{-,b}$}
\label{be_-b}
\end{align}
where, recall,
$$
\beta_1^{}:=-2c\langle |x|_\varepsilon^{-4} x \cdot w, x \cdot (x \cdot \nabla w)|w|^{q-2}\rangle, \qquad \beta_2^{}:=-2c(q-2)\langle |x|_\varepsilon^{-4} (x \cdot w)^2 x \cdot \nabla |w|,|w|^{q-3} \rangle.
$$

We estimate $\langle -b_n\cdot w, \phi^{} \rangle$ as follows.

\begin{lemma}
\label{b_est_lem}
There exist constants $C_i$ \rm{($i=1,2$)} such that
\begin{align*} 
& \langle - b_n\cdot w, \phi \rangle \\
&\leq
|c|(d+3)\frac{q\sqrt{\delta}}{2}G_{q,\chi^2}^{\frac{1}{2}}J_{q}^{\frac{1}{2}} + |c|\frac{q\sqrt{\delta}}{2}\bar{I}_{q,\chi}^{\frac{1}{2}}J_{q}^{\frac{1}{2}}+\biggl(\frac{ q^2\delta}{4} + (q-2)\frac{q\sqrt{\delta}}{2} \biggr)J_{q}
  + C_1\|w\|^q_q + C_2\|w\|_q^{q-2} \|f\|^2_q.
\end{align*}
\end{lemma}

\begin{proof}[Proof of Lemma \ref{b_est_lem}]
For brevity, below $b \equiv b_n$. We have:
\begin{align*}
 \langle-b\cdot w, \phi \rangle & =\langle - \Delta u, |w|^{q-2}(-b\cdot w)\rangle + (q-2) \langle |w|^{q-3} w \cdot \nabla |w|,-b\cdot w\rangle\\
&=:F_1+F_2.
\end{align*}
Set $B_q:=\langle |b \cdot w|^2|w|^{q-2}\rangle$. We have
$$F_2
\leq (q-2) B_{q}^\frac{1}{2} J_{q}^\frac{1}{2}.$$
Next, we bound $F_1$. We represent $-\Delta u=\nabla \cdot (a^\varepsilon-1)\cdot w -\lambda u - b\cdot w +f$,
and evaluate
\begin{align*}
F_1& =\langle \nabla \cdot (a^\varepsilon-1)\cdot w, |w|^{q-2}(-b\cdot w)\rangle  + \langle (-\lambda u - b\cdot w +f),|w|^{q-2}(-b\cdot w) \rangle \\
& (\text{we expand the first term using $\nabla a^\varepsilon=c(d+1)|x|_\varepsilon^{-2}x-2c |x|^2|x|_\varepsilon^{-4}x$}) \\
& = c(d+1)\langle|x|_\varepsilon^{-2}x\cdot w, |w|^{q-2}(-b\cdot w) \rangle \\
&- 2c\langle 
\chi|x|_\varepsilon^{-2}x\cdot w, |w|^{q-2}(-b\cdot w) \rangle \\
& + c \langle |x|_\varepsilon^{-2} x \cdot (x \cdot \nabla w), |w|^{q-2}(-b\cdot w) \rangle \\
& + \langle (-\lambda u - b\cdot w +f),|w|^{q-2}(-b\cdot w) \rangle.
\end{align*}
We bound $F_1$ from above by applying consecutively the following estimates:

\smallskip

$1^\circ$) $\langle|x|_\varepsilon^{-2}x\cdot w, |w|^{q-2}(-b\cdot w) \rangle \leq G_{q,\chi^2}^{\frac{1}{2}} B_{q}^{\frac{1}{2}}$.

\smallskip

$2^\circ$) $\langle \chi|x|_\varepsilon^{-2}x\cdot w, |w|^{q-2}(-b\cdot w) \rangle \leq G_{q,\chi^4}^{\frac{1}{2}} B_{q}^{\frac{1}{2}} \leq G_{q,\chi^2}^{\frac{1}{2}} B_{q}^{\frac{1}{2}}$.

\smallskip

$3^\circ$) $\langle |x|_\varepsilon^{-2} x \cdot (x \cdot \nabla w), |w|^{q-2}(-b\cdot w) \rangle 
\leq \bar{I}_{q,\chi}^{\frac{1}{2}}B_{q}^{\frac{1}{2}}$.

\smallskip

$4^\circ$) $\langle (-\lambda u) , |w|^{q-2} (-b \cdot w) \rangle  \leq \frac{\lambda}{\lambda-\omega_q} B_{q}^\frac{1}{2} \|w\|_q^\frac{q-2}{2}  \|f\|_q$ $\bigl(\text{here } \frac{2}{2-\sqrt{\delta}} < q \Rightarrow \|u_n\|_q \leq (\lambda - \omega_q )^{-1} \|f\|_q \bigr)$.

\smallskip

$5^\circ$) $\langle b \cdot w, |w|^{q-2} b \cdot w \rangle = B_{q} .$

\smallskip

$6^\circ$) $\langle f, |w|^{q-2} (- b \cdot w)\rangle |\leq B_{q}^\frac{1}{2} \|w\|_q^\frac{q-2}{2} \|f\|_q .$

\smallskip

In $4^\circ$) and $6^\circ$ we estimate $B_{q}^\frac{1}{2} \|w\|_q^\frac{q-2}{2} \|f\|_q \leq \varepsilon_0 B_q + \frac{1}{4\varepsilon_0}\|w\|_q^{q-2} \|f\|^2_q$ ($\varepsilon_0>0$).

\smallskip

Therefore,
\begin{align*}
& \langle-b\cdot w, \phi \rangle \\
 &\leq |c|(d+3)G_{q,\chi^2}^{\frac{1}{2}}B_{q}^{\frac{1}{2}} + |c|\bar{I}_{q}^{\frac{1}{2}}B_{q}^{\frac{1}{2}}+B_{q} + (q-2)B_{q}^{\frac{1}{2}}J_{q}^{\frac{1}{2}} + \varepsilon_0 B_q +  C_2(\varepsilon_0)\|w\|_q^{q-2} \|f\|^2_q.
\end{align*}
It is easily seen that $b \in \mathbf{F}_\delta$ 
is equivalent to the inequality
$$
\langle b^2 |\varphi|^2 \rangle \leq \delta \langle |\nabla \varphi|^2\rangle  + \lambda\delta\langle |\varphi|^2 \rangle, \quad \varphi \in W^{1,2}.
$$
Thus,
\begin{equation*}
B_q \leq \|b |w|^\frac{q}{2} \|_2^2 \leq \delta \| \nabla |w|^\frac{q}{2} \|_2^2 + \lambda\delta \|w\|_q^q = \frac{ q^2\delta}{4} J_q + \lambda\delta \|w\|_q^q,  
\end{equation*}
and then selecting $\varepsilon_0>0$ sufficiently small, and noticing that the assumption on $\delta$ in the theorem is a strict inequality, we can and will ignore below the terms multiplied by $\varepsilon_0$. The proof of Lemma \ref{b_est_lem} is completed.
\end{proof}

\textbf{2.~}We estimate $\langle f, \phi \rangle $ in \eqref{be_+b}, \eqref{be_-b} by an evident analogue Lemma \ref{f_est_lem}:
\begin{equation*}
\langle f, \phi \rangle \leq  \varepsilon_0  (I_{q}^{} +  J^{}_{q} + H_q+\|w\|_q^q) + C(\varepsilon_0) \|w\|_q^{q-2} \|f\|^2_q
\end{equation*}
(selecting $\varepsilon_0>0$ sufficiently small so that we will ignore below the terms multiplied by $\varepsilon_0$).

Applying Lemma \ref{b_est_lem} and the last inequality in  \eqref{be_+b}, \eqref{be_-b}, and using
 $\beta_1^{} \leq c\theta \bar{I}_{q,\chi}^{} + c\theta^{-1} G_{q,\chi^2}^{}$, $
|\beta_2^{}| \leq 2|c|(q-2)\bigl(\theta \bar{J}_{q,\chi}^{}  + 4^{-1}\theta^{-1} G_{q,\chi^2}^{}\bigr), \; \theta>0,
$
we obtain: 

If $c>0$, then 
\begin{align}
&\mu \langle |w|^q \rangle + I_q + c(1-\theta)\bar{I}_{q,\chi}+ (q-2)(J_q + c\bar{J}_{q,\chi}) 
 - c\biggl(1+\frac{q-2}{q}d\biggr)H_{q,\chi} + c\big(2(d-1)- \theta^{-1} \big)G_{q,\chi^2} \notag  \\
&+2c\frac{q-2}{q} H_{q,\chi^2} \label{c_b+} \\
& \leq c(d+3)\frac{q\sqrt{\delta}}{2}G_{q,\chi^2}^{\frac{1}{2}}J_{q}^{\frac{1}{2}} + c\frac{q\sqrt{\delta}}{2}\bar{I}_{q,\chi}^{\frac{1}{2}}J_{q}^{\frac{1}{2}}+\biggl(\frac{ q^2\delta}{4} + (q-2)\frac{q\sqrt{\delta}}{2} \biggr)J_{q}
  + C_1\|w\|^q_q + C_2\|w\|_q^{q-2} \|f\|^2_q. \notag
\end{align}
If $-1<c<0$, then (set $s:=|c|$)
\begin{align}
&\mu \langle |w|^q\rangle + I_q^{} - s\bar{I}_{q,\chi}^{} + (q-2)(J_q^{} - s(1+\theta) \bar{J}_{q,\chi}^{}) + s\left(1+(q-2)\frac{d}{q} \right)H_{q,\chi}^{} \notag \\
&-2s\frac{q-2}{q}H_{q,\chi^2}^{} - sd G_{q,\chi^2}^{} -4s G_{q,\chi^2}^{}  + 4sG_{q,\chi^3}^{}  - s(q-2)\frac{1}{4\theta}G_{q,\chi^2}^{} \label{c_b-}\\
&\leq  s(d+3)\frac{q\sqrt{\delta}}{2}G_{q,\chi^2}^{\frac{1}{2}}J_{q}^{\frac{1}{2}} + s\frac{q\sqrt{\delta}}{2}\bar{I}_{q,\chi}^{\frac{1}{2}}J_{q}^{\frac{1}{2}}+\biggl(\frac{ q^2\delta}{4} + (q-2)\frac{q\sqrt{\delta}}{2} \biggr)J_{q}
  + C_1\|w\|^q_q + C_2\|w\|_q^{q-2} \|f\|^2_q. \notag
\end{align}

\textbf{3.~}Employing \eqref{c_b+}, \eqref{c_b-} we will prove the following inequality
\begin{equation}
\label{req_est_40}
\mu \langle |w|^q \rangle + \eta J_q^{} \leq C_1\|w\|_q^{q}+C_2\|w\|_q^{q-2} \|f\|^2_q, \qquad C_i=C_i(\varepsilon_0), \quad i=1,2,
\end{equation}
for some $\eta=\eta(q,d,\varepsilon_0)>0$.

\medskip

\noindent\textbf{Case $c>0$.} In the LHS of \eqref{c_b+} we select $\theta:=\frac{1}{2(d-1)}\;(<1)$.
Consider two subcases:

C1) $1-\frac{cq\sqrt{\delta}}{4} \geq 0$. 
Arguing as in the proof of Theorem \ref{thm1}, using $H_q \geq G_{q,\chi^2}$, we obtain from \eqref{c_b+}:
\begin{align*}
&\mu\langle |w|^q \rangle  + I_q+(q-2)J_q^{} + cM(1)H_q \\
&\leq c(d+3)\frac{q\sqrt{\delta}}{2}H_q^{\frac{1}{2}}J_{q}^{\frac{1}{2}} + c\frac{q\sqrt{\delta}}{2}\bar{I}_{q,\chi}^{\frac{1}{2}}J_{q}^{\frac{1}{2}}+\biggl(\frac{ q^2\delta}{4} + (q-2)\frac{q\sqrt{\delta}}{2} \biggr)J_{q}
  + C_1\|w\|^q_q + C_2\|w\|_q^{q-2} \|f\|^2_q,
\end{align*}
where $M(1):=\bigl(q-1-\frac{1}{2(d-1)}\bigr)\frac{(d-2)^2}{q^2} - \bigl(1+\frac{q-2}{q}(d-2)\bigr)<0$.

 Using the quadratic estimates, we obtain ($\theta_2, \theta_3>0$)
\begin{align*}
&\mu\langle |w|^q \rangle  + I_q+(q-2)J_q^{}  + cM(1)H_q \\
&\leq c(d+3)\frac{q\sqrt{\delta}}{4}(\theta_2J_{q} + \theta^{-1}_2H_q) + c\frac{q\sqrt{\delta}}{4}(\theta_3 \bar{I}_{q,\chi}+\theta_3^{-1}J_{q})+\biggl(\frac{ q^2\delta}{4} + (q-2)\frac{q\sqrt{\delta}}{2} \biggr)J_{q} \\
& + C_1\|w\|^q_q + C_2\|w\|_q^{q-2} \|f\|^2_q.
\end{align*}
We select $\theta_2=\frac{q}{d-2}$, $\theta_3=1$, so
\begin{align*}
&\mu\langle |w|^q \rangle  + I_q+(q-2)J_q^{}  + cM(1)H_q \\
&\leq c(d+3)\frac{q\sqrt{\delta}}{4}\left(\frac{q}{d-2}J_{q} + \frac{d-2}{q}H_q\right) + c\frac{q\sqrt{\delta}}{4}( \bar{I}_{q,\chi}+J_{q})+\biggl(\frac{ q^2\delta}{4} + (q-2)\frac{q\sqrt{\delta}}{2} \biggr)J_{q} \\
& + C_1\|w\|^q_q + C_2\|w\|_q^{q-2} \|f\|^2_q.
\end{align*}
Since  $1-\frac{cq\sqrt{\delta}}{4} \geq 0$, we have $I_{q} - \frac{cq\sqrt{\delta}}{4} \bar{I}_{q,\chi} \geq \big( 1-\frac{cq\sqrt{\delta}}{4}\big)J_q$, so using $J_q \geq \frac{(d-2)^2}{q^2} H_q$ we obtain
\begin{align*}
%\label{req_ineq5}
\mu\langle |w|^q \rangle  +\eta J_q + \bigg[\big(-\eta+q-1\big)\frac{(d-2)^2}{q^2}-\LL_{1}(c,\delta)\bigg] H_q  \leq C_1\|w\|^q_q + C_2\|w\|_q^{q-2} \|f\|^2_q,
\end{align*}
where 
\begin{align*}
\LL_{1}(c,\delta) & =\bigg[c\frac{q\sqrt{\delta}}{4}+c(d+3)\frac{q\sqrt{\delta}}{4}\frac{q}{d-2} + c\frac{q\sqrt{\delta}}{4} + \frac{q^2\delta}{4} + (q-2)\frac{q\sqrt{\delta}}{2} \bigg]\frac{(d-2)^2}{q^2} \notag  \\
& + c\biggl[-M(1) + (d+3)\frac{q\sqrt{\delta}}{4}\frac{d-2}{q}\biggr].
\end{align*}
By the assumptions of the theorem, $\big(-\eta+q-1\big)\frac{(d-2)^2}{q^2}-\LL_{1}(c,\delta) \geq 0$ for all $\eta>0$ sufficiently small. \eqref{req_est_40} is proved.

\smallskip

C2) $1-\frac{cq\sqrt{\delta}}{4}<0$. Arguing as above,  we obtain
\begin{align*}
&\mu\langle |w|^q \rangle  + I_q+(q-2)J_q^{} +c(1-\theta)\bar{I}_{q,\chi} + cM(1)H_q \\
&\leq c(d+3)\frac{q\sqrt{\delta}}{4}\bigg(\frac{q}{d-2}J_{q} + \frac{d-2}{q}H_q\bigg) + c\frac{q\sqrt{\delta}}{4}( \bar{I}_{q,\chi}+J_{q})+\biggl(\frac{ q^2\delta}{4} + (q-2)\frac{q\sqrt{\delta}}{2} \biggr)J_{q} \\
& + C_1\|w\|^q_q + C_2\|w\|_q^{q-2} \|f\|^2_q.
\end{align*}
where $M(1):=\bigl(q-2\bigr)\frac{(d-2)^2}{q^2} - \bigl(1+\frac{q-2}{q}(d-2)\bigr)<0$.

If $1-\frac{1}{2(d-1)}-\frac{q\sqrt{\delta}}{4} < 0$, then clearly
$c\big(1-\frac{1}{2(d-1)}\big)\bar{I}_{q,\chi} - c\frac{q\sqrt{\delta}}{4} \bar{I}_{q,\chi} \geq c\big(1-\frac{1}{2(d-1)} - \frac{q\sqrt{\delta}}{4} \big)I_q$. 
By the assumption of the theorem, 
 $1+c\big(1-\frac{1}{2(d-1)}-\frac{q\sqrt{\delta}}{4}\big) \geq 0$, so in the previous estimate $\big[1+c\big(1-\frac{1}{2(d-1)}-\frac{q\sqrt{\delta}}{4}\big)\big]I_q \geq \big[1+c\big(1-\frac{1}{2(d-1)}-\frac{q\sqrt{\delta}}{4}\big)\big]J_q$, and thus
\begin{align}
\label{req_ineq5}
\mu\langle |w|^q \rangle  +\eta J_q + \bigg[\big(-\eta+q-1\big)\frac{(d-2)^2}{q^2}-\LL_{1}(c,\delta)\bigg] H_q  \leq C_1\|w\|^q_q + C_2\|w\|_q^{q-2} \|f\|^2_q,
\end{align}
where 
\begin{align*}
&\LL_{1}(c,\delta) =\bigg[-c\bigg(1-\frac{1}{2(d-1)}-\frac{q\sqrt{\delta}}{4}\bigg) + c\frac{q\sqrt{\delta}}{4} + c(d+3)\frac{q\sqrt{\delta}}{4}\frac{q}{d-2} + \frac{q^2\delta}{4} + (q-2)\frac{q\sqrt{\delta}}{2} \bigg]\frac{(d-2)^2}{q^2} \notag \\
& + c\biggl[-M(1) + (d+3)\frac{q\sqrt{\delta}}{4}\frac{d-2}{q}\biggr]
\end{align*}
(another representation for $\LL_{1}(c,\delta)$).
By the assumptions of the theorem, $\big(-\eta+q-1\big)\frac{(d-2)^2}{q^2}-\LL_{1}(c,\delta) \geq 0$ for all $\eta>0$ sufficiently small. \eqref{req_est_40} is proved.

If  $1-\frac{1}{2(d-1)}-\frac{q\sqrt{\delta}}{4} \geq 0$, then clearly $I_q+c\big(1-\frac{1}{2(d-1)}-\frac{q\sqrt{\delta}}{4}\big)\bar{I}_{q,\chi} \geq J_q+c\big(1-\frac{1}{2(d-1)}-\frac{q\sqrt{\delta}}{4}\big)\bar{J}_{q,\chi}$. Arguing as above, we obtain \eqref{req_ineq5} and therefore \eqref{req_est_40}.

\medskip

\noindent\textbf{Case $-1<c<0$.} In the LHS of \eqref{c_b-} we select $\theta=\frac{1}{2}\frac{q}{d-2}$. Arguing as in the proof of Theorem \ref{thm1}, we obtain from \eqref{c_b-}:
\begin{align}
&\mu \langle |w|^q\rangle + I_q^{} - s\bar{I}_{q,\chi}^{} + (q-2)(J_q^{} - s(1+\theta) \bar{J}_{q,\chi}^{}) + sM(1)G_q \\
&\leq  s(d+3)\frac{q\sqrt{\delta}}{2}G_{q}^{\frac{1}{2}}J_{q}^{\frac{1}{2}} + s\frac{q\sqrt{\delta}}{2}\bar{I}_{q,\chi}^{\frac{1}{2}}J_{q}^{\frac{1}{2}}+\biggl(\frac{ q^2\delta}{4} + (q-2)\frac{q\sqrt{\delta}}{2} \biggr)J_{q}
  + C_1\|w\|^q_q + C\|w\|_q^{q-2} \|f\|^2_q,\notag
\end{align}
where 
$M(1):=-d+1+\frac{1}{2}\frac{(q-2)(d-2)}{q}<0$. In the RHS of \eqref{c_b-} we have used $G_q \geq G_{q,\chi^2}$.

Further ($\theta_2$, $\theta_3>0$),
\begin{align*}
&\mu \langle |w|^q\rangle + I_q^{} - s\bar{I}_{q,\chi}^{} + (q-2)(J_q^{} - s(1+\theta) \bar{J}_{q,\chi}^{}) + sM(1)G_q \\
&\leq  s(d+3)\frac{q\sqrt{\delta}}{4}(\theta_2J_{q} + \theta^{-1}_2G_q) + s\frac{q\sqrt{\delta}}{4}(\theta_3 \bar{I}_{q,\chi}+\theta_3^{-1}J_{q})+\biggl(\frac{ q^2\delta}{4} + (q-2)\frac{q\sqrt{\delta}}{2} \biggr)J_{q} \\
& + C_1\|w\|^q_q + C_2\|w\|_q^{q-2} \|f\|^2_q.
\end{align*}
Selecting $\theta_2=\frac{q}{d-2}$, $\theta_3=1$ and using the inequalities $I_q \geq \bar{I}_{q,\chi}$, $J_q \geq \bar{J}_{q,\chi}$, 
we obtain
\begin{align*}
&\mu \langle |w|^q\rangle + \bigg[1-s\bigg(1+\frac{q\sqrt{\delta}}{4}\bigg)\bigg]I_q^{} \\
& + \biggl[(q-2)\bigg(1- s\bigg(1+\frac{1}{2}\frac{q}{d-2}\bigg)\bigg)-s\frac{q\sqrt{\delta}}{4}-s(d+3)\frac{q\sqrt{\delta}}{4}\frac{q}{d-2} -\frac{q^2\delta}{4} - (q-2)\frac{q\sqrt{\delta}}{2}\biggr]J_q \\
&+ s\biggl[M(1) -(d+3)\frac{q\sqrt{\delta}}{4}\frac{d-2}{q}\biggr]G_q \leq  C_1\|w\|^q_q + C_2\|w\|_q^{q-2} \|f\|^2_q.
\end{align*}
By the assumptions of the theorem, $
1-s\big(1+\frac{q\sqrt{\delta}}{4}\big) \geq 0.
$
So, using the inequalities $J_q \leq I_q$, $J_q \geq \frac{(d-2)^2}{q^2}G_q$, 
 we arrive at
\begin{align*}
\mu\langle |w|^q \rangle  +\eta J_q + \bigg[\big(-\eta+q-1\big)\frac{(d-2)^2}{q^2}-\LL_{2}(s,\delta)\bigg] G_q  \leq C_1\|w\|^q_q + C_2\|w\|_q^{q-2} \|f\|^2_q,
\end{align*}
where 
\begin{align*}
& \LL_{2}(s,\delta) \\
&=\bigg[s\bigg(1+\frac{q\sqrt{\delta}}{4}\bigg)  + (q-2)s\bigg(1+\frac{1}{2}\frac{q}{d-2}\bigg)+s\frac{q\sqrt{\delta}}{4}+s(d+3)\frac{q\sqrt{\delta}}{4}\frac{q}{d-2} +\frac{q^2\delta}{4} + (q-2)\frac{q\sqrt{\delta}}{2}\biggr]\frac{(d-2)^2}{q^2} \\
&+ s\biggl[-M(1) +(d+3)\frac{q\sqrt{\delta}}{4}\frac{d-2}{q}\biggr]. 
\end{align*}
By the assumptions of the theorem, $\big(-\eta+q-1\big)\frac{(d-2)^2}{q^2} - \LL_{2}(s,\delta) \geq 0$ for all $\eta>0$ sufficiently small. \eqref{req_est_40} is proved.

\smallskip

\textbf{4.~}The Sobolev Embedding Theorem and Theorem \ref{conv_appendix}(\textit{i})  now yield \eqref{reg_double_star} (cf.\,the proof of Theorem \ref{thm1}, step 4).

\smallskip

\textit{Proof of {\rm(}iii{\rm)}}. Let $q=2$, $d \geq 3$.
Recall that since $b \in \mathbf{F}_\delta$, then 
$b \in \mathbf{F}_{\delta_1}(A)$, where $\delta_1:=\delta$ if $c>0$, and $\delta_1:=\delta (1+c)^{-2}$  if $-1<c<0$
By our assumptions, $\delta_1<1$, so $\Lambda_2(a^\varepsilon,b_n)=A^\varepsilon_2 + b_n\cdot\nabla$ are well defined on $L^2$.
Following the proof of Theorem \ref{thm1}(\textit{iii}),
we obtain: for $c>0$
\begin{align*}
&\langle [\nabla,A_2^\varepsilon]_-u,w \rangle = -\beta_1 - cH_{2,\chi} + 2c(d-1)G_{2,\chi^2},\\
&\mu \|w\|_2^2 +I_2 +c\bar{I}_{2,\chi} - \beta_1 -cH_{2,\chi} +2c(d-1)G_{2,\chi^2} = \langle -b_n \cdot w,-\nabla \cdot w\rangle + \langle f, -\nabla \cdot w \rangle\\
&\beta_1=-2c\langle |x|^{-4}x\cdot w, x\cdot(x\cdot \nabla w)\rangle;
\end{align*}
for $c<0$
\begin{align*}
&\langle [\nabla,A_2^\varepsilon]_-u,w \rangle = \frac{1}{2}\beta_2 - cH_{2,\chi} + cdG_{2,\chi^2}, \quad \beta_2=0,\\
&\mu \|w\|_2^2 +I_2 +c\bar{I_2} -cH_{2,\chi} +cdG_{2,\chi^2} = \langle -b_n \cdot w,-\nabla \cdot w\rangle + \langle f, -\nabla \cdot w \rangle, \quad I_2\geq\bar{I_2}, \\
&\mu \|w\|_2^2 + (1-|c|)I_2 + |c|H_{2,\chi}-|c|dG_{2,\chi^2} \leq \langle -b_n \cdot w,-\nabla \cdot w\rangle + \langle f,-\nabla \cdot w  \rangle.
\end{align*}
Now, applying Lemma \ref{b_est_lem} and arguing as in step 3 above, we arrive at $\sup_{\varepsilon>0, n \geq 1}I_2(u^{\varepsilon,n})  \leq K\|f\|_2$, and so  (by passing to the limit $\varepsilon \downarrow 0$, using Theorem \ref{conv_appendix}, we arrive at $I_2(u)  \leq K\|f\|_2$ $\Rightarrow$
$u \in W^{2,2}$.

The proof of Theorem \ref{thm2} is completed. \qed

\section{Proof of Theorem \ref{thm3}}

Recall that a vector field $b:\mathbb R^d \rightarrow \mathbb R^d$ belongs to $\mathbf F_{\delta_1}(A)$, $\delta_1>0$, the class of form-bounded vector fields (with respect to $A \equiv A_2$), if $b_a^2:=b \cdot a^{-1} \cdot b
 \in L^1_\loc$ and there exists a constant $\lambda=\lambda_{\delta_1}>0$ such that 
$\|b_a(\lambda + A)^{-\frac{1}{2}}\|_{2\rightarrow 2}\leq \sqrt{\delta_1}.$

We will need the following auxiliary results.
Recall:  $(\nabla a)_k = \sum_{i=1}^d (\partial_{x_i} a_{ik})$, $1 \leq k \leq d$.

\begin{lemma}
\label{nabla_a_lem}
$
\nabla a=c(d-1)|x|^{-2}x\in \mathbf{F}_{\delta_0}(A)$, where $\delta_0:=4\big(\frac{d-1}{d-2}\frac{c}{1+c}\big)^2.$
\end{lemma}

\begin{proof}
It is easy to see that $b:=\nabla a=c(d-1)|x|^{-2}x$,
$a^{-1} = I - \frac{c}{c+1} |x|^{-2} x \otimes x$ and $b_a^2:=b \cdot a^{-1} \cdot b = \frac{[(d-1) c ]^2}{c+1} |x|^{-2}.$ Now, that $b \in \mathbf{F}_\delta(A)$
is immediate from the following Hardy-type inequality:
\begin{equation}
\label{h_type}
(c+1) \frac{(d-2)^2}{4} \| |x|^{-1} h \|_2^2 \leq \langle \nabla h \cdot a \cdot \nabla \bar{h} \rangle, \qquad  h \in W^{1,2}(\mathbb R^d).
\tag{$\star$}
\end{equation}
It remains to prove \eqref{h_type}.
Since $\langle \phi, x \cdot \nabla \phi \rangle = - \frac{d}{2} \langle \phi, \phi \rangle, \; \phi \in C_c^\infty,$ we have
\begin{equation}
\label{blacktriangle}
\langle \phi, -\nabla \cdot (a-1) \cdot \nabla \phi \rangle = c \big(\| x \cdot \nabla (|x|^{-1} \phi) \|_2^2 - (d-1) \| |x|^{-1} \phi \|_2^2 \big) 
\end{equation}
Next, the following inequality (with the sharp constant) is valid:
\begin{equation}
\label{whitetriangle}
\| x \cdot \nabla f \|_2 \geq \frac{d}{2} \|f \|_2, \quad \quad ( f \in D(\mathcal{D}) ), \\
\end{equation}
where the operator $\mathcal{D} = (\mathcal{D} \upharpoonright C_c^\infty )^{clos}_{L^2 \to L^2},$ $\mathcal{D} \upharpoonright C_c^\infty = \frac{\sqrt{-1}}{2} ( x \cdot \nabla + \nabla \cdot x ),$ is selfadjoint. 

Indeed, by the Spectral Theorem, $\|(\mathcal{D} - \zeta)^{-1} \|_{2\to 2} =\frac{1}{|\Imag \zeta |}$ for $\Real \zeta = 0$, and hence

\[
\| x \cdot \nabla f \|_2 = \| \frac{1}{2}( x \cdot \nabla + \nabla \cdot x -d )f \|_2 = \|(\mathcal{D} - \sqrt{- 1} \; \frac{d}{2}) f \|_2 \geq \frac{d}{2} \|f \|_2, \quad \quad (f \in C_c^\infty ).
\]
\eqref{whitetriangle} is proved. 

Let $c>0.$ By \eqref{blacktriangle} and \eqref{whitetriangle},
\[
\langle \phi, -\nabla \cdot (a-1) \cdot \nabla \phi \rangle \geq c \frac{(d-2)^2}{4} \| |x|^{-1} \phi \|_2^2 \quad \quad ( \phi \in C^\infty_c ).
\]
$(\star)$ follows now from the equality $\langle \phi, - \nabla \cdot a \cdot \nabla \phi \rangle =  \langle \phi, - \nabla \cdot (a-1) \cdot \nabla \phi \rangle + \langle \phi, - \Delta \phi \rangle$ and Hardy's inequality $\langle \phi, -\Delta \phi \rangle \geq \frac{(d-2)^2}{4} \| |x|^{-1} \phi \|_2^2.$
Finally, the obvious inequality $(1+c) \langle \phi, -\Delta \phi \rangle \geq \langle \phi, - \nabla \cdot a \cdot \nabla \phi \rangle$ shows that the constant in $(\star)$ is sharp. 

If $-1 < c < 0,$  $(\star)$ is a trivial consequence of Hardy's inequality.
\end{proof}

If $0<c<d-2$, then $\nabla a \in \mathbf{F}_\delta(A)$ with $\delta<4$ by Lemma \ref{nabla_a_lem}, so $\Lambda_q(a,\nabla a)$ is well defined for all $q \in \big](1-\frac{d-1}{d-2}\frac{c}{1+c})^{-1},\infty\big[$, see \cite[Theorem 3.2]{KiS}.

If $-1<c<0$, then $\Lambda_q(a,\nabla a)$ is well defined for all $q \in \big]1,\infty\big[$ by Theorem \ref{thm:A1} (there take $b=0$). 
We have proved assertion (\textit{i}) of the theorem.

\smallskip

In order to  prove assertion (\textit{ii}), we will need the following  result.
Set 
\[
|x|_\varepsilon:=\sqrt{|x|^2+\varepsilon}, \; \varepsilon>0, \quad \chi:=|x|^2|x|_\varepsilon^{-2}.
\]

\begin{lemma}
\label{approx_prop} 
Set $a^\varepsilon(x):=I+c |x|_\varepsilon^{-2}x^{t}\cdot x$, $A^\varepsilon \equiv [-\nabla \cdot a^\varepsilon \cdot \nabla\upharpoonright C_c^\infty]^{\rm clos}_{2 \to 2}.$ If $d\geq4$, $ -1 < c \leq \frac{d-3}{2}$, or if $d=3$, $-1<c<0$,  then
\[
\nabla a^\varepsilon \in \mathbf{F}_{\delta_0}(A^\varepsilon) \quad \text{ with } \delta_0=4\bigg(\frac{d-1}{d-2}\frac{c}{1+c}\bigg)^2.
\] 
\end{lemma}

\begin{proof}
1.~First, let $c>0$. Note that $(a^\varepsilon)^{-1}(x)=I-\frac{c\chi}{1+c\chi}|x|^{-2} x \otimes x$, $(\nabla a^\varepsilon)=c\chi(d+1-2 \chi)|x|^{-2} x$ and 
\smallskip

(a) $(\nabla a^\varepsilon) \cdot (a^\varepsilon)^{-1} \cdot (\nabla a^\varepsilon)=\frac{(c\chi)^2[(d+1-2\chi)]^2}{1+c\chi} |x|^{-2}.$

(b) $\langle -\nabla \cdot a^\varepsilon \cdot \nabla h,h\rangle = \langle (\nabla h)^2 \rangle +c \langle |x|_\varepsilon^{-2} (x \cdot \nabla h)^2\rangle 
\geq \frac{(d-2)^2}{4} \langle |x|^{-2}h^2\rangle + c \big\langle\big(\frac{d^2}{4}-(d+2)\chi+3\chi^2\big)\chi|x|^{-2}h^2\big\rangle$
$h \in C_c^\infty$, see \eqref{hardy_type_ineq}.

\smallskip

Combining (a) and (b), we obtain that $\nabla a^{\varepsilon} \in \mathbf{F}_{\delta^\varepsilon}(A^\varepsilon)$ for any $\delta^\varepsilon$ such that
\begin{align*}
&\delta^\varepsilon \biggl\langle\biggl[\frac{(d-2)^2}{4} + c\biggl(\frac{d^2}{4}-(d+2)\chi+3\chi^2\biggr)\chi\biggr]|x|^{-2}h^2\biggr\rangle  \\
&\geq \biggl\langle\frac{(c\chi)^2[(d+1-2\chi)]^2}{1+c\chi} |x|^{-2}h^2 \biggr\rangle, \quad h \in C_c^\infty;
\end{align*}
we can take
$$
\delta^\varepsilon:=\sup_{0\leq t \leq 1}\frac{(ct)^2[d+1-2 t]^2}{(1+c t)\biggl[\frac{(d-2)^2}{4} + ct\biggl(\frac{d^2}{4}-(d+2) t + 3 t^2\biggr)\biggr]}
$$

Let us show that
\begin{equation*}
\delta^\varepsilon = 4\biggl(\frac{d-1}{d-2}\frac{c}{1+c}\biggr)^2\equiv \delta_0,
\end{equation*}
which would imply that $\nabla a^{\varepsilon} \in \mathbf{F}_{\delta_0}(A^\varepsilon)$, as claimed.

Note that $\frac{d^2}{4}-(d+2) t + 3 t^2 \geq \frac{(d-2)^2}{4}$ for all $0\leq  t\leq 1$ and $d\geq 4$. Thus
\[
\delta_0 \leq \delta^\varepsilon \leq 4\sup_{0\leq t \leq 1}\frac{c^2[d+1-2 t]^2t^2}{(d-2)^2(1 + ct)^2}= 4 \bigg[\frac{c}{d-2}\sup_{0\leq t \leq 1}\frac{(d+1-2 t)t}{1 + ct}\bigg]^2 =\delta_0.
\]

2. Let $-1<c<0$. Then $\langle -\nabla \cdot a^\varepsilon \cdot \nabla h,h\rangle \geq \frac{(d-2)^2}{4}(1+c)\langle |x|^{-2}h^2\rangle$, so by (a) above $\nabla a^{\varepsilon} \in \mathbf{F}_{\delta^\varepsilon}(A^\varepsilon)$ for any $\delta^\varepsilon$ such that
$$
\delta^\varepsilon \frac{(d-2)^2}{4}(1+c) \langle |x|^{-2}h^2\rangle \geq \biggl\langle\frac{c^2[(d+1-2\chi)]^2}{1+c\chi}\chi^2 |x|^{-2}h^2 \biggr\rangle, \quad h \in C_c^\infty;
$$
we can take
$$
\delta^\varepsilon:=\sup_{0 \leq t \leq 1}\frac{c^2(d+1-2t)^2t^2}{\frac{(d-2)^2}{4}(1+c)(1+ct)}.
$$
Finally, since $1+c \leq 1+ct$,
\[
\delta_0 \leq \delta_\varepsilon \leq 4 \big[\frac{c}{(1+c)(d-2)}\sup_{0<t<1}(d+1-2 t)t\big]^2 =\delta_0.
\]  
\end{proof}

\medskip

\textbf{1.~}We start the proof of assertion (\textit{ii}) of the theorem. Let $d \geq 4$. We follow closely the proof of Theorem \ref{thm1}. 
Set $u^\varepsilon = (\mu + \Lambda_q(a^\varepsilon, \nabla a^{\varepsilon}))^{-1} f$, $0 \leq f \in C_c^1$. Since $a^{\varepsilon} \in C^\infty$, we have $ 0 \leq u^\varepsilon \in W^{3,q}$. Below
$$
w \equiv w^\varepsilon:=\nabla u^\varepsilon, \quad \phi:=- \nabla \cdot (w |w|^{q-2}),
$$
$$
I_q:=\langle (\nabla_r w)^2 |w|^{q-2}  \rangle, \quad
J_q:=\langle (\nabla |w|)^2 |w|^{q-2}\rangle,$$
$$
\bar{I}_{q,\chi}:= \langle  \bigl( x \cdot \nabla w \bigr)^2 \chi|x|^{-2} |w|^{q-2}\rangle, \quad
\bar{J}_{q,\chi}:= \langle   (x \cdot \nabla |w|)^2 \chi|x|^{-2}|w|^{q-2} \rangle,
$$
$$H_{q,\chi}:=\langle \chi|x|^{-2}|w|^q \rangle, \qquad G_{q,\chi^2}:=\langle \chi^2|x|^{-4} (x \cdot w)^2
|w|^{q-2}\rangle,
$$
where $\chi=|x|^2|x|_\varepsilon^{-2}$. 
We will need

\begin{lemma}[The basic equalities, non-divergence form]
\label{basic_ineq_lemma_nondiv}
\begin{align}
&\mu \langle |w|^q \rangle + I_q^{} + c\bar{I}_{q,\chi}^{} + (q-2)(J_q^{} + c\bar{J}_{q,\chi}^{}) 
 - c\frac{d(d-1)}{q}H_{q,\chi}^{}  \notag \\
&+2c\frac{2d+1}{q}H_{q,\chi^2}-\frac{8c}{q}H_{q,\chi^3} =  \beta_1^{} + \langle f, \phi^{} \rangle
\tag{${\rm BE}^{{\rm nd}}_+$},
\label{be_+_nondiv}
\end{align}
\begin{align}
&\mu \langle |w|^q \rangle + I_q^{} + c\bar{I}_{q,\chi}^{} + (q-2)(J_q^{} + c\bar{J}_{q,\chi}^{})
 - c\frac{d(d-1)}{q}H_{q,\chi}^{} - c(d-2)G_{q,\chi^2} \notag \\
&+2c\frac{2d+1}{q}H_{q,\chi^2}-\frac{8c}{q}H_{q,\chi^3} -4c\varepsilon \langle |x|_\varepsilon^{-6}(x \cdot w)^2|w|^{q-2}\rangle = - \frac{1}{2}\beta_2^{}+ \langle f, \phi^{} \rangle,
\tag{${\rm BE}^{{\rm nd}}_-$} 
\label{be_-_nondiv}
\end{align}
where 
$$
\beta_1^{}:=-2c\langle |x|_\varepsilon^{-4} x \cdot w, x \cdot (x \cdot \nabla w)|w|^{q-2}\rangle, \qquad \beta_2^{}:=-2c(q-2)\langle |x|_\varepsilon^{-4} (x \cdot w)^2 x \cdot \nabla |w|,|w|^{q-3} \rangle.
$$
\end{lemma}
\begin{proof}[Proof of Lemma \ref{basic_ineq_lemma_nondiv}]
We modify the proof of Lemma \ref{basic_ineq_lemma}.
In the left-hand side of \eqref{be_+}, \eqref{be_-} we have the extra term $\langle (\nabla a^\varepsilon) \cdot w,-\nabla \cdot (w|w|^{q-2})\rangle$, which we evaluate as follows:
\begin{align*}
&\langle (\nabla a^\varepsilon) \cdot w,-\nabla \cdot (w|w|^{q-2})\rangle \\
& \text{(we integrate by parts)} \\
&=c(d-1)\bigl(H_{q,\chi}^{} + \langle |x|_\varepsilon^{-2} x \cdot \nabla |w|,|w|^{q-1}\rangle - 2G_{q,\chi^2}^{} \bigr) \\
&+2c\varepsilon\bigl(\langle |x|_\varepsilon^{-4}|w|^q\rangle + \langle |x|_\varepsilon^{-4} x \cdot \nabla |w|,|w|^{q-1}\rangle  - 4 \langle |x|_\varepsilon^{-6}(x \cdot w)^2 |w|^{q-2}\rangle \bigr).
\end{align*}
Note that
$$\langle |x|_\varepsilon^{-2} x \cdot \nabla |w|,|w|^{q-1}\rangle= \frac{1}{q}\langle |x|_\varepsilon^{-2} x \cdot \nabla |w|^q\rangle=-\frac{d-2}{q} H_{q,\chi}^{}- \frac{2}{q}\varepsilon \langle |x|_\varepsilon^{-4}|w|^q\rangle,
$$
$$\langle |x|_\varepsilon^{-4}x\cdot\nabla|w|,|w|^{q-1} \rangle=\frac{1}{q}\langle |x|_\varepsilon^{-4}x \cdot \nabla |w|^q\rangle=-\frac{1}{q}\langle |w|^q \nabla \cdot (x|x|_\varepsilon^{-4})\rangle
=-\frac{d-4}{q}\langle |x|_\varepsilon^{-4}|w|^q\rangle  -\frac{4}{q}\varepsilon \langle |x|_\varepsilon^{-6}|w|^q\rangle.
$$
Thus,
\begin{align*}
&\big\langle (\nabla a^\varepsilon) \cdot w,-\nabla \cdot (w|w|^{q-2})\big\rangle\\
&=c(d+1)\left(1-\frac{d}{q}\right)H_{q,\chi}+c\biggl(-2+\frac{2}{q}(2d+3) \biggr)H_{q,\chi^2}-\frac{8c}{q}H_{q,\chi^3}\\
& -2c(d-1)G_{q,\chi^2} -8c\varepsilon \big\langle|x|_\varepsilon^{-6}(x \cdot w)^2|w|^{q-2}\big\rangle.
\end{align*}
The latter, added to the left-hand side of \eqref{be_+}, \eqref{be_-} yields \eqref{be_+_nondiv}, \eqref{be_-_nondiv}.
\end{proof}

\textbf{2.~}We estimate from above the term $\langle f, \phi \rangle$ in the right-hand side of \eqref{be_+_nondiv}, \eqref{be_-_nondiv} employing an evident analogue of Lemma \ref{f_est_lem}:
\begin{equation}
\label{f_phi_est_again}
\langle f, \phi \rangle \leq  \varepsilon_0  (I_{q}^{} +  J^{}_{q} + H_q) + C(\varepsilon_0) \|w\|_q^{q-2} \|f\|^2_q.
\end{equation}
Again we choose $\varepsilon_0>0$ so small that in the estimates below we can ignore the terms multiplied by $\varepsilon_0$.

\medskip

\textbf{3.}~We will use \eqref{be_+_nondiv}, \eqref{be_-_nondiv} and \eqref{f_phi_est_again} to establish the inequality
\begin{equation}
\label{req_est4__}
\mu \langle |w|^q \rangle + \eta J_q^{} \leq C\|w\|_q^{q-2} \|f\|^2_q, \qquad C=C(\varepsilon_0), \qquad \eta=\eta(q,d,\varepsilon_0)>0.
\end{equation}

\textit{\textbf{Case $c>0$}}. By the assumptions of the theorem, $c<\frac{d-3}{2} \wedge \frac{d-2}{q-d+2}$.
In \eqref{be_+_nondiv}, we estimate
$$\beta_1 \leq c\theta \bar{I}_{q,\chi} + c\theta^{-1}G_{q,\chi^2}, \;\; \theta >0,$$ 
and then
apply \eqref{f_phi_est_again} to obtain
\begin{align*}
&\mu\langle |w|^q\rangle + I_q +c(1-\theta)\bar{I}_{q,\chi} + (q-2)J_q + c(q-2)\bar{J}_{q,\chi} -c\frac{d(d-1)}{q}H_{q,\chi} \\
&+ 2c\frac{2d+1}{q}H_{q,\chi^2}  - \frac{8c}{q} H_{q,\chi^3} -\frac{c}{\theta}G_{q,\chi^2} \leq C\|w\|_q^{q-2} \|f\|^2_q.
\end{align*}
We exclude the case $0<\theta\leq 1$ by noting that $\bar{I}_{q,\chi} \geq \frac{(d-2)^2}{q^2}G_{q,\chi^2}$
and $f(\theta)=(1-\theta)\frac{(d-2)^2}{q^2}-\frac{1}{\theta}$ achieves its maximum at $\theta= \frac{q}{d-2} >1$.

Let $\theta >1$. Clearly we have to assume now that $1+c(1-\theta)>0$. Since $I_q+c(1-\theta)\bar{I}_{q,\chi} \geq (1+c(1-\theta))I_q \geq (1+c(1-\theta))J_q$ and $H_{q,\chi^2} \geq G_{q,\chi^2}$ we have
\begin{align*}
&\mu\langle |w|^q\rangle +(q-1+c(1-\theta))J_q + c(q-2)\bar{J}_{q,\chi} -c\frac{d(d-1)}{q}H_{q,\chi} \\
&+ 2c\frac{2d+1}{q}H_{q,\chi^2}  - \frac{8c}{q} H_{q,\chi^3} -\frac{c}{\theta}H_{q,\chi^2} \leq C\|w\|_q^{q-2} \|f\|^2_q.
\end{align*} 
Using $\bar{J}_{q,\chi} \geq \frac{4}{q^2}\left(\frac{d^2}{4}H_{q,\chi}-(d+2)H_{q,\chi^2}+ 3H_{q,\chi^3}\right)$,  see \eqref{hardy_type_ineq}, we obtain
\begin{align*}
&\mu\langle |w|^q\rangle + (q-1+c(1-\theta))J_q + c(q-2)\frac{4}{q^2}\left(\frac{d^2}{4}H_{q,\chi}-(d+2)H_{q,\chi^2} + 3H_{q,\chi^3} \right) -c\frac{d(d-1)}{q}H_{q,\chi} \\
&+ 2c\frac{2d+1}{q}H_{q,\chi^2}  - \frac{8c}{q} H_{q,\chi^3} -\frac{c}{\theta}H_{q,\chi^2} \leq C\|w\|_q^{q-2} \|f\|^2_q.
\end{align*} 
Thus, by $J_q \geq \frac{(d-2)^2}{q^2}\langle |x|^{-2}|w|^q\rangle$, for all $\eta>0$ sufficiently small,
\begin{align*}
\mu\langle |w|^q \rangle 
+ \eta J_q  + \biggl\langle\biggl[(-\eta+q-1+c(1-\theta))\frac{(d-2)^2}{q^2}+c M(\chi)  \biggr]|x|^{-2}|w|^q\biggr\rangle \leq C\|w\|_q^{q-2} \|f\|^2_q,
\end{align*} 
where 
$$
M(\chi):=\biggl[(q-2)\frac{4}{q^2}\left(\frac{d^2}{4}-(d+2)\chi + 3\chi^2 \right) -\frac{d(d-1)}{q}
+ 2\frac{2d+1}{q}\chi  - \frac{8}{q} \chi^2 -\frac{1}{\theta}\chi\biggr]\chi.
$$
Select $\theta:=\frac{q}{d-2}$. (Motivation: estimating the terms involving $\theta$ from below by 
$\bigl[-c\theta \frac{(d-2)^2}{q^2} - \frac{c}{\theta}\bigr]H_q$ and maximizing the latter in $\theta$, we arrive at $\theta=\frac{q}{d-2}$.) Then, since $c<\frac{d-2}{q-d+2}$, we have  $1+c(1-\theta)>0$.
Elementary arguments show that
\[
\min_{0\leq t \leq 1}M(t)=M(1)<0,
\]
and so
\begin{align*}
\mu\langle |w|^q \rangle 
+ \eta J_q  + \biggl[(-\eta+q-1)\frac{(d-2)^2}{q^2}-c \ell_1^{{\rm nd}} \biggr]H_q \leq C\|w\|_q^{q-2} \|f\|^2_q,
\end{align*} 
where $\ell^{{\rm nd}}_1:=\frac{q+d-2}{q^2}(d-2)$.
By the assumption $c<\frac{d-3}{2}$ of the theorem, there exists $\eta>0$ such that $(-\eta+q-1)\frac{(d-2)^2}{q^2}-c \ell_1^{{\rm nd}} \geq 0$. Thus \eqref{req_est4__} is proved.

\smallskip

\textit{\textbf{Case $-1<c<0$.}} By the assumptions of the theorem, $-\bigl(1+\frac{1}{4}\frac{q}{d-2}\frac{q-2}{q-1}\frac{q-2}{q+d-3}\bigr)^{-1}<c<0$. ~Set $s:=|c|$. In \eqref{be_-_nondiv} we estimate ($\theta>0$)
$$
|\beta_2^\varepsilon| \leq 2s(q-2)\bigl(\theta\bar{J}_q  + 4^{-1}\theta^{-1} G_{q,\chi^2}\bigr),
$$
obtaining
\begin{align*}
&\mu \langle |w|^q\rangle + I_q - s\bar{I}_{q,\chi} + (q-2)(J_q - s(1+\theta) \bar{J}_{q,\chi}) + s\frac{d(d-1)}{q}H_{q,\chi} \\
&  - 2s \frac{2d+1}{q}H_{q,\chi^2}+\frac{8s}{q}H_{q,\chi^3} +s\left(d+2-(q-2)\frac{1}{4\theta}\right) G_{q,\chi^2}- 4sG_{q,\chi^3}  \leq  C\|w\|_q^{q-2} \|f\|^2_q.
\end{align*}
Then by the obvious inequalities $I_q - s\bar{I}_{q,\chi} \geq (1-s)J_q$ and $J_q - s(1+\theta) \bar{J}_{q,\chi} \geq (1-s(1+\theta))J_q$,
\begin{align}
\label{req_ineq4}
\tag{$\star$}
&\mu \langle |w|^q\rangle + \bigl(q-1-s- s(q-2)(1+\theta)\bigr)J_q
 + s\frac{d(d-1)}{q}H_{q,\chi} \\
& - 2s \frac{2d+1}{q}H_{q,\chi^2}+\frac{8s}{q}H_{q,\chi^3} +s\left(d+2-(q-2)\frac{1}{4\theta}\right) G_{q,\chi^2}  - 4sG_{q,\chi^3}  \leq  C\|w\|_q^{q-2} \|f\|^2_q. \notag
\end{align}
Note that
$
d(d-1)-2(2d+1)t + 8t^2 \geq 0, \; (d \geq 3, \; 0 \leq t \leq 1)$,
and so 
$$
\frac{d(d-1)}{q}H_{q,\chi}  - 2 \frac{2d+1}{q}H_{q,\chi^2}+\frac{8}{q}H_{q,\chi^3} \geq \frac{d(d-1)}{q}G_{q,\chi}  - 2 \frac{2d+1}{q}G_{q,\chi^2}+\frac{8}{q}G_{q,\chi^3}.
$$
Therefore, we obtain from \eqref{req_ineq4}
\begin{align*}
&\mu \langle |w|^q\rangle + \bigl(q-1-s- s(q-2)(1+\theta)\bigr)J_q
 + s\frac{d(d-1)}{q}G_{q,\chi} \\
& - 2s \frac{2d+1}{q}G_{q,\chi^2}+\frac{8s}{q}G_{q,\chi^3} +s\left(d+2-(q-2)\frac{1}{4\theta}\right) G_{q,\chi^2}  - 4sG_{q,\chi^3}  \leq  C\|w\|_q^{q-2} \|f\|^2_q, \notag
\end{align*}  
i.e.
\begin{align*}
\mu \langle |w|^q\rangle + \bigl[q-1-s- s(q-2)(1+\theta)\bigr]J_q
- s \bigl\langle M(\chi) |x|^{-4}(x \cdot w)^2|w|^{q-2}\bigr\rangle \leq  C\|w\|_q^{q-2} \|f\|^2_q,
\end{align*}
where 
$$
M(\chi):=q^{-1}\big[\mathfrak a\chi^2+\mathfrak b\chi+\mathfrak c_0\big]\chi.
$$
$$
\mathfrak a:=4(q-2), \quad \mathfrak b:=2(2d+1)-q\biggl(d+2-(q-2)\frac{1}{4\theta}\biggr), \quad \mathfrak c_0:=-d(d-1).
$$
Select $\theta:=\frac{1}{4}\frac{q}{d-2}\frac{q-2}{q+d-3}$ if $q>2$. Then $M(0)=M(1)=\max_{0\leq t\leq1} M(t) = 0$. This is the best possible choice of $\theta$. (Selecting a larger $\theta$, so that $\max_{0\leq t\leq1} M(t)<0$, decreases the term $[\dots]J_q$. On the other hand, selecting a smaller $\theta$, so that $\max_{0\leq t\leq1} M(t)>0$, leads to constraints on $c$ which are sub-optimal, i.e.\,which can be improved by selecting a larger $\theta$.)

Note that $q-1-s- s(q-2)(1+\theta) >0$ by the assumptions of the theorem. Thus, 
\begin{align*}
&\mu \langle |w|^q\rangle  +\big(q-1-s- s(q-2)(1+\theta)\big) J_q \leq  C\|w\|_q^{q-2} \|f\|^2_q,
\end{align*}
and hence \eqref{req_est4__} is proved for $q > 2$. 

We are left to treat the case $d=4$ and $q=2$. Note that the proof above still works. See also a proof of (iii) below. 

\smallskip

\textbf{4.~}For $d \geq 4$, the Sobolev Embedding Theorem and Theorem \ref{conv_appendix}(\textit{ii}) (with $\delta=0$)  now yield estimates \eqref{reg_est000} and convergence \eqref{good_sol_conv}. The proof of (\textit{ii}) is completed.

\smallskip

\textit{Proof of {\rm(}\textit{iii}{\rm)}}. Let $q=2$, $d \geq 3$. If $c<0$, then we can argue as in steps 1-3 obtaining
$$
\sup_{\varepsilon>0}I_2(u^\varepsilon) \leq K\|f\|_2, 
\quad \text{ and so } I_2(u)  \leq K\|f\|_2 \quad \Rightarrow u \in W^{2,2}.$$

Now, let $c>0$. By Lemma \ref{nabla_a_lem},
$
\nabla a \in \mathbf{F}_{\delta_0}(A)$, $\delta_0=4\big(\frac{d-1}{d-2}\frac{c}{1+c}\big)^2.$ 
Since $c<\frac{d-2}{d}$, we have $\delta_0<1$, and so $\Lambda_2(a,\nabla a)$ is well defined. By the Miyadera Perturbation Theorem and Theorem \ref{thm1}, $D(\Lambda_2(a,\nabla a))=D(A_2) \subset W^{2,2}$, and
$u:=(\mu+\Lambda_2(a,\nabla a))^{-1}f$, $\mu>0$, $f \in L^2$, belongs to $W^{2,2}$. Multiplying $(\mu  + A_2 + \nabla a\cdot\nabla) u^{}= f $ by $\phi_{m}:=- E_m\nabla \cdot w$, where $w:=\nabla u$, $E_m=(1-m^{-1}\Delta)^{-1}$, $m \geq 1$, and integrating by parts we have (omitting the summation sign in the repeated indices):
\begin{align*}
%\label{integr_id}
\tag{$\star$}
\mu \langle |w|^2 \rangle + \langle a \cdot \nabla w_r, E_m \nabla w_r\rangle + \langle-(\nabla_r a) \cdot w, E_m\nabla w_r)\rangle  +\langle \nabla a\cdot w, \phi_{m}^{} \rangle  = \langle f, \phi_{m}^{} \rangle,
\end{align*}

Now we pass in ($\star$) to the limit $m \rightarrow \infty$.
 
Then following closely the proof of \eqref{be_+_nondiv} for $q=2$ we obtain: 
\begin{align*}
\mu \langle |w|^2 \rangle + I_2^{} + c\bar{I}_{2}^{} 
 - \frac{c}{2}(d-2)(d-3)H_{2} = \beta  + \langle f, \phi^{} \rangle,
\end{align*}
where 
$I_2:= \langle \nabla w_r,\nabla w_r \rangle$, $\bar{I}_2:= \langle  \bigl( x \cdot \nabla w \bigr)^2 |x|^{-2} \rangle$, $H_2:=\langle |x|^{-2}|w|^2 \rangle$, $\beta:=-2c\langle |x|^{-4} x \cdot w, x \cdot (x \cdot \nabla w)\rangle$. 

Using the inequalities $\beta \leq c\bar{I}_2 +cH_2$, $\frac{(d-2)^2}{4}H_2\leq I_2$, we have
\[
\mu \langle |w|^2 \rangle +\bigg[1-\frac{4}{(d-2)^2}\bigg(1+\frac{(d-2)(d-3)}{2}\bigg)c\bigg]I_2 \leq \langle f, \phi^{} \rangle.
\]
The proof of (iii) follows.

\smallskip

The proof of Theorem \ref{thm3} is completed. \qed

\section{Proof of Theorem \ref{thm4}}

 We follow closely the proofs of Theorems \ref{thm2}, \ref{thm3}. 

\smallskip

\textit{Proof of (i)}.
It is easily seen that if $b \in \mathbf{F}_\delta$, then 
$b \in \mathbf{F}_{\delta_1}(A)$, where $\delta_1:=\delta$ if $c>0$, and $\delta_1:=\delta (1+c)^{-2}$  if $-1<c<0$. Further, by Lemma \ref{nabla_a_lem}, 
$\nabla a=c(d-1)|x|^{-2}x\in \mathbf{F}_{\delta_0}(A)$, where $\delta_0:=4\big(\frac{d-1}{d-2}\frac{c}{1+c}\big)^2.$
Now, set as in the formulation of assertion (\textit{i}),
$$
\sqrt{\delta_2}:=\left\{
\begin{array}{ll}
\sqrt{\delta_1}+\sqrt{\delta_0}, & 0<c<d-2, \\
\sqrt{\delta_1}, & -1<c<0.
\end{array}
\right.
$$
For $c>0$, we have by our assumption $\delta_2<4$ if $c>0$, so by \cite[Theorem 3.2]{KiS} the formal differential expression $-a \cdot \nabla^2 + b\cdot \nabla$ ($\equiv -\nabla \cdot a \cdot \nabla +  (\nabla a)\cdot \nabla +  b \cdot \nabla$) has an operator realization  $\Lambda_q(a,\nabla a + b)$ in $L^q$, $q \in \big[\frac{2}{2-\sqrt{\delta_2}}, \infty\big[$, as the (minus) generator of a  positivity preserving $L^\infty$ contraction quasi contraction $C_0$ semigroup; moreover, $(\mu+\Lambda_q(a,b))^{-1}$ is well defined on $L^q$ for all $\mu>\frac{\lambda \delta_2}{2(q-1)}$. In case $c<0$, we apply Theorem \ref{thm:A1}.
This completes the proof of (\textit{i}).

\smallskip

\textit{Proof of {\rm(}ii{\rm)}}. Let $d \geq 4$. Set
$
a^\varepsilon:=I+c|x|^{-2}_\varepsilon x \otimes x$, $|x|_\varepsilon:=\sqrt{|x|^2+\varepsilon}$,  $\varepsilon>0$. Put $A^\varepsilon=A(a^\varepsilon)$. It is clear that $b \in \mathbf{F}_{\delta_1}(A^\varepsilon)$ for all $\varepsilon>0$.

Let $\mathbf 1_n$ denote the indicator of  $\{x \in \mathbb R^d \mid  \; |x| \leq n,  |b(x)| \leq n \}$, and set $b_n := \gamma_{\epsilon_n} \ast \mathbf 1_n b \in C^\infty$, where $\gamma_{\epsilon}$ is the K.\,Friedrichs mollifier, $\epsilon_n \downarrow 0$.
Since our assumptions on $\delta$ and thus $\delta_1$ involve strict inequalities only, we can select $\epsilon_n \downarrow 0$ so that   $b_n \in \mathbf{F}_{\delta_1}(A^\varepsilon)$, $\varepsilon>0$, $n \geq 1$. Next, note that by the assumptions of the theorem $-1<c<\frac{d-3}{2}$, and hence by Lemma \ref{approx_prop}, $\nabla a^\varepsilon \in \mathbf{F}_{\delta_0}(A^\varepsilon)$, $\varepsilon>0$. 
Thus, in view of the discussion above, $(\mu+\Lambda_q(a^\varepsilon,\nabla a^\varepsilon + b_n))^{-1}$ is well defined on $L^q$, $\mu>\frac{\lambda\delta_2}{2(q-1)}$, $\varepsilon>0$, $n \geq 1$.
Here $\Lambda_q(a^\varepsilon, \nabla a^\varepsilon + b_n)=-\nabla \cdot a^\varepsilon \cdot \nabla +(\nabla a^\varepsilon)\cdot\nabla + b_n \cdot \nabla$, $D(\Lambda_q(a^\varepsilon, \nabla a^\varepsilon + b_n))=W^{2,q}$.

Set $u \equiv u^{\varepsilon,n} = (\mu + \Lambda_q(a^\varepsilon, \nabla a^{\varepsilon} + b_n))^{-1} f$, $0 \leq f \in C_c^1$. Then $u \in W^{3,q}$. Below
$$
w \equiv w^{\varepsilon,n}:=\nabla u^{\varepsilon,n}, \quad \phi:=- \nabla \cdot (w |w|^{q-2}),
$$
$$
I_q:=\langle (\nabla_r w)^2 |w|^{q-2}  \rangle, \quad
J_q:=\langle (\nabla |w|)^2 |w|^{q-2}\rangle,$$
$$
\bar{I}_{q,\chi}:= \langle \bigl( x \cdot \nabla w \bigr)^2 \chi|x|^{-2} |w|^{q-2} \rangle, \quad
\bar{J}_{q,\chi}:= \langle  (x \cdot \nabla |w|)^2 \chi|x|^{-2} |w|^{q-2} \rangle,
$$
$$H_{q,\chi}:=\langle \chi|x|^{-2}|w|^q \rangle, \qquad G_{q,\chi^2}:=\langle \chi^2|x|^{-4} (x \cdot w)^2
|w|^{q-2}\rangle,
$$
where $\chi=|x|^2|x|_\varepsilon^{-2}$.

\smallskip

\textbf{1.~}We repeat the proof of Lemma \ref{basic_ineq_lemma_nondiv}, where in the right-hand side of \eqref{be_+_nondiv}, \eqref{be_-_nondiv}
we now get an extra term $\langle -b_n\cdot w, \phi^{} \rangle$:
\begin{align}
&\mu \langle |w|^q \rangle + I_q^{} + c\bar{I}_{q,\chi}^{} + (q-2)(J_q^{} + c\bar{J}_{q,\chi}^{}) 
 - c\frac{d(d-1)}{q}H_{q,\chi}^{}  \notag \\
&+2c\frac{2d+1}{q}H_{q,\chi^2}-\frac{8c}{q}H_{q,\chi^3} =  \beta_1^{} + \langle -b_n\cdot w, \phi \rangle + \langle f, \phi^{} \rangle
\tag{${\rm BE}^{{\rm nd}}_{+,b}$},
\label{be_+b_nondiv}
\end{align}
\begin{align}
&\mu \langle |w|^q \rangle + I_q^{} + c\bar{I}_{q,\chi}^{} + (q-2)(J_q^{} + c\bar{J}_{q,\chi}^{})
 - c\frac{d(d-1)}{q}H_{q,\chi}^{} - c(d-2)G_{q,\chi^2} \notag \\
&+2c\frac{2d+1}{q}H_{q,\chi^2}-\frac{8c}{q}H_{q,\chi^3} -4c\varepsilon \langle |x|_\varepsilon^{-6}(x \cdot w)^2|w|^{q-2}\rangle = - \frac{1}{2}\beta_2^{}+ \langle -b_n\cdot w, \phi \rangle + \langle f, \phi^{} \rangle,
\tag{${\rm BE}^{{\rm nd}}_{-,b}$} 
\label{be_-b_nondiv}
\end{align}
where 
$$
\beta_1^{}:=-2c\langle |x|_\varepsilon^{-4} x \cdot w, x \cdot (x \cdot \nabla w)|w|^{q-2}\rangle, \qquad \beta_2^{}:=-2c(q-2)\langle |x|_\varepsilon^{-4} (x \cdot w)^2 x \cdot \nabla |w|,|w|^{q-3} \rangle.
$$

\textbf{2.~}By Lemma \ref{b_est_lem},
\begin{align*} 
& \langle - b_n\cdot w, \phi \rangle \\
&\leq
|c|(d+3)\frac{q\sqrt{\delta}}{2}G_{q,\chi^2}^{\frac{1}{2}}J_{q}^{\frac{1}{2}} + |c|\frac{q\sqrt{\delta}}{2}\bar{I}_{q,\chi}^{\frac{1}{2}}J_{q}^{\frac{1}{2}}+\biggl(\frac{ q^2\delta}{4} + (q-2)\frac{q\sqrt{\delta}}{2} \biggr)J_{q}
  + C_1\|w\|^q_q + C_2\|w\|_q^{q-2} \|f\|^2_q.
\end{align*}
Next, by an evident analogue of Lemma \ref{f_est_lem},
\begin{equation*}
\langle f, \phi \rangle \leq  \varepsilon_0  (I_{q}^{} +  J^{}_{q} + H_q+\|w\|_q^q) + C(\varepsilon_0) \|w\|_q^{q-2} \|f\|^2_q.
\end{equation*}
Again we choose $\varepsilon_0>0$ so small that in the estimates below we can ignore the terms multiplied by $\varepsilon_0$.

Applying the last two inequalities in  \eqref{be_+b_nondiv}, \eqref{be_-b_nondiv}, and using
 $\beta_1^{} \leq c\theta \bar{I}_{q,\chi}^{} + c\theta^{-1} G_{q,\chi^2}^{}$, $
|\beta_2^{}| \leq 2|c|(q-2)\bigl(\theta \bar{J}_{q,\chi}^{}  + 4^{-1}\theta^{-1} G_{q,\chi^2}^{}\bigr),
$
we obtain: 

If $c>0$, then:
\begin{align}
&\mu \langle |w|^q \rangle + I_q^{} + c(1-\theta)\bar{I}_{q,\chi}^{} + (q-2)(J_q^{} + c\bar{J}_{q,\chi}^{}) 
 - c\frac{d(d-1)}{q}H_{q,\chi}^{}  \notag \\
&+2c\frac{2d+1}{q}H_{q,\chi^2}-\frac{8c}{q}H_{q,\chi^3} - \frac{c}{\theta} G_{q,\chi^2} \label{c_b+_nondiv}\\
&\leq  c(d+3)\frac{q\sqrt{\delta}}{2}G_{q,\chi^2}^{\frac{1}{2}}J_{q}^{\frac{1}{2}} + c\frac{q\sqrt{\delta}}{2}\bar{I}_{q,\chi}^{\frac{1}{2}}J_{q}^{\frac{1}{2}}+\biggl(\frac{ q^2\delta}{4} + (q-2)\frac{q\sqrt{\delta}}{2} \biggr)J_{q}
  + C_1\|w\|^q_q + C_2\|w\|_q^{q-2} \|f\|^2_q. \notag
\end{align}
If $-1<c<0$, then (set $s:=|c|$):
\begin{align}
&\mu \langle |w|^q \rangle + I_q^{} -s\bar{I}_{q,\chi}^{} + (q-2)(J_q^{} -s(1+\theta)\bar{J}_{q,\chi}^{})
 +s\frac{d(d-1)}{q}H_{q,\chi}^{}  \notag \\
&-2s\frac{2d+1}{q}H_{q,\chi^2}+\frac{8s}{q}H_{q,\chi^3} +s\left(d+2-(q-2)\frac{1}{4\theta}\right) G_{q,\chi^2}- 4sG_{q,\chi^3} \label{c_b-_nondiv} \\
&\leq  s(d+3)\frac{q\sqrt{\delta}}{2}G_{q,\chi^2}^{\frac{1}{2}}J_{q}^{\frac{1}{2}} + s\frac{q\sqrt{\delta}}{2}\bar{I}_{q,\chi}^{\frac{1}{2}}J_{q}^{\frac{1}{2}}+\biggl(\frac{ q^2\delta}{4} + (q-2)\frac{q\sqrt{\delta}}{2} \biggr)J_{q}
  + C_1\|w\|^q_q + C_2\|w\|_q^{q-2} \|f\|^2_q. \notag
\end{align}

\textbf{3.~}We will use \eqref{c_b+_nondiv}, \eqref{c_b-_nondiv} to prove the following inequality
\begin{equation}
\label{req_est_40_nondiv}
\mu \langle |w|^q \rangle + \eta J_q^{} \leq C_1\|w\|_q^{q-2}+C_2\|w\|_q^{q-2} \|f\|^2_q, \qquad C_i=C_i(\varepsilon_0), \quad i=1,2,
\end{equation}
for some $\eta=\eta(q,d,\varepsilon_0)>0$.

\textit{\textbf{Case $c>0$.}}~In \eqref{c_b+_nondiv}, select $\theta= \frac{q}{d-2} >1$.
By the assumptions of the theorem, $1+c(1-\theta)>0$. Since $I_q+c(1-\theta)\bar{I}_{q,\chi} \geq (1+c(1-\theta))I_q$ and $H_{q,\chi^2} \geq G_{q,\chi^2}$, we have
\begin{align*}
&\mu\langle |w|^q\rangle +(1+c(1-\theta))I_q + (q-2)(J_q^{} + c\bar{J}_{q,\chi}^{}) -c\frac{d(d-1)}{q}H_{q,\chi} \\
&+ 2c\frac{2d+1}{q}H_{q,\chi^2}  - \frac{8c}{q} H_{q,\chi^3} -\frac{c}{\theta}H_{q,\chi^2} \\
&\leq  c(d+3)\frac{q\sqrt{\delta}}{2}G_{q,\chi^2}^{\frac{1}{2}}J_{q}^{\frac{1}{2}} + c\frac{q\sqrt{\delta}}{2}\bar{I}_{q,\chi}^{\frac{1}{2}}J_{q}^{\frac{1}{2}}+\biggl(\frac{ q^2\delta}{4} + (q-2)\frac{q\sqrt{\delta}}{2} \biggr)J_{q}
  + C_1\|w\|^q_q + C_2\|w\|_q^{q-2} \|f\|^2_q.
\end{align*}
Arguing as in the proof of Theorem \ref{thm3}, we arrive at
\begin{align*}
&\mu\langle |w|^q\rangle +(1+c(1-\theta))I_q + (q-2) J_q^{} +cM(1) H_q \\
&\leq  c(d+3)\frac{q\sqrt{\delta}}{2}G_{q,\chi^2}^{\frac{1}{2}}J_{q}^{\frac{1}{2}} + c\frac{q\sqrt{\delta}}{2}\bar{I}_{q,\chi}^{\frac{1}{2}}J_{q}^{\frac{1}{2}}+\biggl(\frac{ q^2\delta}{4} + (q-2)\frac{q\sqrt{\delta}}{2} \biggr)J_{q}
  + C_1\|w\|^q_q + C_2\|w\|_q^{q-2} \|f\|^2_q,
\end{align*}
where $M(1)=-2\frac{(d-2)^2}{q^2}$.
Using $I_q \geq \bar{I}_{q,\chi}$, $H_q \geq G_{q,\chi^2}$ in the RHS, and applying the standard quadratic estimates, we obtain ($\theta_2, \theta_3>0$),
\begin{align*}
&\mu\langle |w|^q\rangle +(1+c(1-\theta))I_q + (q-2)J_q^{} +cM(1) H_q \\
&\leq c(d+3)\frac{q\sqrt{\delta}}{4}(\theta_2J_{q} + \theta^{-1}_2H_q) + c\frac{q\sqrt{\delta}}{4}(\theta_3 I_{q}+\theta_3^{-1}J_{q})+\biggl(\frac{ q^2\delta}{4} + (q-2)\frac{q\sqrt{\delta}}{2} \biggr)J_{q} \\
& + C_1\|w\|^q_q + C_2\|w\|_q^{q-2} \|f\|^2_q.
\end{align*}
We select $\theta_2=\frac{q}{d-2}$, $\theta_3=1$. By the assumptions of the theorem, $1+c\big(1-\theta-\frac{q\sqrt{\delta}}{4}\big) \geq 0$, so that $\big(1+c(1-\theta -\frac{q\sqrt{\delta}}{4} )\big)I_q \geq \big(1+c(1-\theta -\frac{q\sqrt{\delta}}{4} )\big)J_q$. Thus, we arrive at
\begin{align*}
&\mu\langle |w|^q\rangle +\biggl[ q-1+c\biggl(1-\theta -\frac{q\sqrt{\delta}}{2}\biggr)  -  c(d+3)\frac{q\sqrt{\delta}}{4}\frac{q}{d-2} - \frac{ q^2\delta}{4} - (q-2)\frac{q\sqrt{\delta}}{2} \biggr]J_q \\
&+\biggl[cM(1) - c(d+3)\frac{q\sqrt{\delta}}{4}\frac{d-2}{q} \biggr] H_q  \leq C_1\|w\|^q_q + C_2\|w\|_q^{q-2} \|f\|^2_q.
\end{align*}
So, by $J_q \geq \frac{(d-2)^2}{q^2}H_q$,
\begin{align*}
&\mu\langle |w|^q\rangle +\eta J_q+\biggl[(-\eta + q-1)\frac{(d-2)^2}{q^2} - \LL_1^{\rm nd}(c,\delta)\biggr]H_q \leq C_1\|w\|^q_q + C_2\|w\|_q^{q-2} \|f\|^2_q,
\end{align*}
where
\begin{align*}
&\LL_1^{\rm nd}(c,\delta)=\biggl[-c\biggl(1-\theta -\frac{q\sqrt{\delta}}{2}\biggr)  +  c(d+3)\frac{q\sqrt{\delta}}{4}\frac{q}{d-2} + \frac{ q^2\delta}{4} + (q-2)\frac{q\sqrt{\delta}}{2} \biggr]\frac{(d-2)^2}{q^2} \\
&+c\biggl[-M(1) + (d+3)\frac{q\sqrt{\delta}}{4}\frac{d-2}{q} \biggr]. 
\end{align*}
By the assumptions of the theorem, $\big(-\eta+q-1\big)\frac{(d-2)^2}{q^2}-\LL_{1}(c,\delta) \geq 0$ for all $\eta>0$ sufficiently small. \eqref{req_est_40_nondiv} is proved.

\smallskip

\textit{\textbf{Case $-1<c<0$.}} Following the proof of Theorem \ref{thm3}, we obtain from \eqref{c_b-_nondiv}
\begin{align*}
&\mu \langle |w|^q\rangle + (1-s)I_q + (q-2)\bigl(1- s(1+\theta)\bigr)J_q
 + s\frac{d(d-1)}{q}G_{q,\chi} \\
& - 2s \frac{2d+1}{q}G_{q,\chi^2}+\frac{8s}{q}G_{q,\chi^3} +s\left(d+2-(q-2)\frac{1}{4\theta}\right) G_{q,\chi^2}  - 4sG_{q,\chi^3} \\
& \leq s(d+3)\frac{q\sqrt{\delta}}{2}G_{q,\chi^2}^{\frac{1}{2}}J_{q}^{\frac{1}{2}} + s\frac{q\sqrt{\delta}}{2}\bar{I}_{q,\chi}^{\frac{1}{2}}J_{q}^{\frac{1}{2}}+\biggl(\frac{ q^2\delta}{4} + (q-2)\frac{q\sqrt{\delta}}{2} \biggr)J_{q}
  + C_1\|w\|^q_q + C_2\|w\|_q^{q-2} \|f\|^2_q,
\end{align*}
In the RHS, we use $\frac{q^2}{(d-2)^2}J_q \geq G_{q,\chi^2}$, $I_q \geq \bar{I}_{q,\chi}$, $\frac{1}{2}(I_q+J_q) \geq I_{q}^{\frac{1}{2}}J_{q}^{\frac{1}{2}}$
\begin{align*}
&\mu \langle |w|^q\rangle + (1-s)I_q + (q-2)\bigl(1- s(1+\theta)\bigr)J_q
 + s\frac{d(d-1)}{q}G_{q,\chi} \\
& - 2s \frac{2d+1}{q}G_{q,\chi^2}+\frac{8s}{q}G_{q,\chi^3} +s\left(d+2-(q-2)\frac{1}{4\theta}\right) G_{q,\chi^2}  - 4sG_{q,\chi^3} \\
& \leq s(d+3)\frac{q\sqrt{\delta}}{2} \frac{q}{d-2} J_{q} + s\frac{q\sqrt{\delta}}{4}(I_q+J_q)+\biggl(\frac{ q^2\delta}{4} + (q-2)\frac{q\sqrt{\delta}}{2} \biggr)J_{q}
  + C_1\|w\|^q_q + C_2\|w\|_q^{q-2} \|f\|^2_q.
\end{align*}
Arguing as in the proof of Theorem \ref{thm3}, and selecting $\theta:=\frac{1}{4}\frac{q}{d-2}\frac{q-2}{q+d-3}$, we arrive at
\begin{align*}
&\mu \langle |w|^q\rangle + (1-s)I_q + (q-2)\bigl(1- s(1+\theta)\bigr)J_q \\
& \leq s(d+3)\frac{q\sqrt{\delta}}{2} \frac{q}{d-2} J_{q} + s\frac{q\sqrt{\delta}}{4}(I_q+J_q)+\biggl(\frac{ q^2\delta}{4} + (q-2)\frac{q\sqrt{\delta}}{2} \biggr)J_{q}
  + C_1\|w\|^q_q + C_2\|w\|_q^{q-2} \|f\|^2_q.
\end{align*}
By the assumptions of the theorem, $1-s\big(1+\frac{q\sqrt{\delta}}{4}\big) \geq 0$. Therefore, since $J_q \leq I_q$,
\begin{align*}
&\mu \langle |w|^q\rangle + \bigg[ q-1 -s- s(q-2)(1+\theta) - s(d+3)\frac{q\sqrt{\delta}}{2} \frac{q}{d-2}- s\frac{q\sqrt{\delta}}{2} - \frac{ q^2\delta}{4} - (q-2)\frac{q\sqrt{\delta}}{2}\biggr]J_q \\
& \leq C_1\|w\|^q_q + C_2\|w\|_q^{q-2} \|f\|^2_q.
\end{align*}
By the assumptions of the theorem, $q-1 -s - s(q-2)(1+\theta) - s(d+3)\frac{q\sqrt{\delta}}{2} \frac{q}{d-2}- s\frac{q\sqrt{\delta}}{2} - \frac{ q^2\delta}{4} - (q-2)\frac{q\sqrt{\delta}}{2}>0$.  Hence \eqref{req_est_40_nondiv} is proved.

\smallskip

\textbf{4.~}For $d \geq 4$, the Sobolev Embedding Theorem and Theorem \ref{conv_appendix}(\textit{ii})  now yield \eqref{reg_double_star}. We have proved (\textit{ii}).

\smallskip

\textit{Proof of (iii)}. Let $q=2$, $d \geq 3$. If $c<0$, then we can argue as in steps 1-3 obtaining
$$
\sup_{\varepsilon>0, n}I_2(u^{\varepsilon,n}) \leq K\|f\|_2, 
\quad \text{ and so } I_2(u)  \leq K\|f\|_2 \quad \Rightarrow u \in W^{2,2}.$$ 

%then Lemma \ref{approx_prop} doesn't impose any additional constraints on $c$, and we can argue as in steps 1-3 obtaining
%$\sup_{\varepsilon>0}I_2(u^\varepsilon) \leq K'\|f\|_2$, and so $I_2(u)  \leq K'\|f\|_2$ $\Rightarrow$
%$u \in W^{2,2}$. 

Now, let $c>0$. We have
$
b+\nabla a \in \mathbf{F}_{\delta_2}(A)$, $\sqrt{\delta_2}:=\sqrt{\delta} + 2\frac{d-1}{d-2}\frac{c}{1+c}$
(cf.\;beginning of the proof). By the assumptions of the theorem, $\delta_2<1$,
and so $\Lambda_2(a,\nabla a)$ is well defined. By the Miyadera Perturbation Theorem and Theorem \ref{thm1}, $D(\Lambda_2(a,\nabla a + b))=D(A_2) \subset W^{2,2}$, and
$u:=(\mu+\Lambda_2(a,\nabla a + b))^{-1}f$, $\mu>0$, $f \in L^2$, belongs to $W^{2,2}$. Multiplying $(\mu  + A_2 + (\nabla a + b)\cdot\nabla) u^{}= f $ by $\phi_{m}:=- E_m\nabla \cdot w$, $E_m=(1-m^{-1}\Delta)^{-1}$, $m \geq 1$ and integrating by parts we have (omitting the summation sign in the repeated indices):
\begin{align*}
\tag{$\star$}
\mu \langle |w|^2 \rangle + \langle a \cdot \nabla w_r, E_m \nabla w_r\rangle + \langle-(\nabla_r a) \cdot w, E_m\nabla w_r)\rangle  +\langle \nabla a\cdot w, \phi_{m}^{} \rangle  = \langle -b\cdot w,\phi_m\rangle + \langle f, \phi_{m}^{} \rangle,
\end{align*}

Now we pass in ($\star$) to the limit $m \rightarrow \infty$. We obtain an analogue of \eqref{be_+b_nondiv} for $q=2$: 
\begin{align*}
\mu \langle |w|^2 \rangle + I_2^{} + c\bar{I}_{2}^{} 
 - \frac{c}{2}(d-2)(d-3)H_{2} = \beta  + \langle -b\cdot w,\phi\rangle +  \langle f, -\nabla \cdot w \rangle
\end{align*}
where 
$I_2:= \langle \nabla w_r,\nabla w_r \rangle$, $\bar{I}_2:= \langle  \bigl( x \cdot \nabla w \bigr)^2 |x|^{-2} \rangle$, $H_2:=\langle |x|^{-2}|w|^2 \rangle$, $\beta:=-2c\langle |x|^{-4} x \cdot w, x \cdot (x \cdot \nabla w)\rangle$. 

Using the inequalities $\beta \leq c\bar{I}_2 +cH_2$, $\frac{(d-2)^2}{4}H_2\leq I_2$, we have
\[
\mu \langle |w|^2 \rangle +\bigg[1-\frac{4}{(d-2)^2}\bigg(1+\frac{(d-2)(d-3)}{2}\bigg)c\bigg]I_2 \leq \langle -b\cdot w,\phi\rangle + \langle f, \phi^{} \rangle,
\]
so by 
\begin{align*}
|\langle b \cdot w,\phi\rangle| & \leq  c(d+3)\frac{\sqrt{\delta}}{2}\left(\frac{2}{d-2}J_{2} + \frac{d-2}{2}H_q\right) + c\frac{\sqrt{\delta}}{2}( I_2+J_{2})+\delta J_{2} \qquad (J_2 \leq I_2) \\
& \leq \biggl(2c\sqrt{\delta}\frac{d+3}{d-2} + c\sqrt{\delta}+\delta \biggr)I_2.
\end{align*}
Therefore, 
\[
\mu \langle |w|^2 \rangle +\bigg[1-\frac{4}{(d-2)^2}\bigg(1+\frac{(d-2)(d-3)}{2}\bigg)c - c\sqrt{\delta}\biggl(2\frac{d+3}{d-2}+1 \biggr)-\delta\bigg]I_2 \leq \langle f, \phi^{} \rangle,
\]
where the coefficient of $I_2$ is positive by the assumptions of the theorem.
The proof of (\textit{iii}) follows.

\smallskip

The proof of Theorem \ref{thm4} is completed. \qed

\appendix

\section{}
\label{appendix_existence}

\setcounter{theorem}{0}
\renewcommand{\thetheorem}{\Alph{section}.\arabic{theorem}}

The following theorem is essentially a special case of \cite[Theorem 3.2]{KiS}.

\begin{theorem}
\label{thm:A1}
Let $d \geq $3. Let $a=I+c|x|^{-2}x \otimes x$, $
a^\varepsilon:=I+c|x|^{-2}_\varepsilon x \otimes x,
$ $|x|_\varepsilon^2:=|x|^2+\varepsilon$, $\varepsilon>0$. 
Set $a_n:=a^{\varepsilon_n}$, $\varepsilon_n \downarrow 0$. Let $b \in \mathbf{F}_{\delta_1}(A)$, $\delta_1>0$. 
Let $\mathbf 1_n$ denote the indicator of  $\{x \in \mathbb R^d \mid  \; |x| \leq n,  |b(x)| \leq n \}$, and set $b_n := \gamma_{\epsilon_n} \ast \mathbf 1_n b$, where $\gamma_{\epsilon}$ is the K.\,Friedrichs mollifier, $\epsilon_n \downarrow 0$. 

{\rm(\textit{i})} 
Let $\delta_1 < 4$.  Then $-\nabla \cdot a \cdot \nabla + b \cdot \nabla$ has an operator realization  $\Lambda_r(a,b)$ in $L^r$, $r > \frac{2}{2-\sqrt{\delta_1}}$, as the (minus) generator of a positivity preserving, $L^\infty$ contraction, quasi contraction $C_0$ semigroup on $L^r$, 
\[
e^{-t \Lambda_r(a,b)} := s\mbox{-}L^r\mbox{-}\lim_{n\rightarrow \infty} e^{-t \Lambda_r(a,b_n)}.
\]

{\rm(\textit{ii})} Set $\sqrt{\delta_0}:=2 \frac{d-1}{d-2}\frac{|c|}{1+c}$,
$$
\sqrt{\delta_2}:=\left\{
\begin{array}{ll}
\sqrt{\delta_1}+\sqrt{\delta_0}, & c>0, \\
\sqrt{\delta_1},  & -1<c<0.
\end{array}
\right.
$$
Assume that $\delta_2 < 4$.  Then $-\nabla \cdot a \cdot \nabla + (\nabla a + b) \cdot \nabla$ has an operator realization  $\Lambda_r(a,\nabla a+b)$ in $L^r$, $r > \frac{2}{2-\sqrt{\delta_2}}$, as the (minus) generator of a positivity preserving, $L^\infty$ contraction, quasi contraction $C_0$ semigroup on $L^r$, 
\[
e^{-t \Lambda_r(a,\nabla a + b)} := s\mbox{-}L^r\mbox{-}\lim_{n\rightarrow \infty} e^{-t \Lambda_r(a,\nabla a_n + b_n)}.
\]
\end{theorem}

\begin{proof} 
Below we use that $b \in \mathbf{F}_{\delta_1}(A)$ is equivalent to: 
$$
\langle b \cdot a^{-1} \cdot b, |\varphi|^2 \rangle \leq \delta_1 \langle \nabla \varphi \cdot a \cdot \nabla \bar{\varphi}\rangle  + \lambda\delta_1  \langle |\varphi|^2 \rangle, \quad \varphi \in W^{1,2}.
$$
 Since our assumption on $\delta_1$ involves a strict inequality, we can select $\epsilon_n \downarrow 0$ so that   $b_n \in \mathbf{F}_{\delta_1}(A)$, $n \geq 1$. 

It is not difficult to see that $\nabla a_n \in \mathbf{F}_{\delta_0}(A)$, $n \geq 1$.

We will conduct the proof of (\textit{ii}), for $-1<c<0$, $r \geq 2$. The proof of (\textit{i}), and of the remaining cases in (\textit{ii}), is similar.

We will need the following elementary estimate. Since $\nabla a_n=c(d-1)|x|_{\varepsilon_n}^{-2} x + 2c\varepsilon_n |x|_{\varepsilon_n}^{-4}x$ , we have, for all $v \in L_+^\infty\cap W^{1,2}$ (write $\varepsilon = \varepsilon_n$),
$$
\langle \nabla a_n \cdot \nabla v,v^{r-1}\rangle=  c(d-1)\langle |x|_\varepsilon^{-2}x \cdot \nabla v,v^{r-1}\rangle  + 2c\varepsilon\langle |x|_\varepsilon^{-4}x \cdot \nabla v,v^{r-1}\rangle,
$$
$$
\frac{r}{2}\langle |x|_\varepsilon^{-2} x \cdot \nabla v,v^{r-1}\rangle=\langle |x|_\varepsilon^{-2} x \cdot \nabla v^{\frac{r}{2}},v^{\frac{r}{2}}\rangle=-\frac{d-2}{2} \langle |x|_\varepsilon^{-2}v^r\rangle- \varepsilon \langle |x|_\varepsilon^{-4}v^r\rangle,
$$
$$\frac{r}{2}\varepsilon \langle |x|_\varepsilon^{-4}x\cdot\nabla v,v^{r-1} \rangle=\varepsilon\langle |x|_\varepsilon^{-4}x\cdot\nabla v^{\frac{r}{2}},v^{\frac{r}{2}} \rangle=-\frac{d-4}{2}\varepsilon\langle |x|_\varepsilon^{-4}v^r\rangle  -2\varepsilon^2 \langle |x|_\varepsilon^{-6}v^r\rangle.
$$
Thus, since $c<0$ and $v \geq 0$,  
\begin{equation}
\tag{$\bullet$}
\label{pos}
\langle \nabla a_n \cdot \nabla v,v^{r-1}\rangle \geq 0 \quad \text{for all $d\geq 3$.}
\end{equation}

\smallskip

1)~By the standard theory,
$-\nabla \cdot a \cdot \nabla + (\nabla a_n + b_n) \cdot \nabla$ has an operator realization  $\Lambda_r(a,\nabla a_n+b_n)$ in $L^r$ as the (minus) generator of a positivity preserving, $L^\infty$ contraction, quasi contraction $C_0$ semigroup on $L^r$. Moreover, $u \equiv u_n:=e^{-t\Lambda_r(a,\nabla a_n + b_n)}f=e^{-t\Lambda_{2r}(a,\nabla a_n + b_n)}f$, $f \in L^\infty \cap L^1_+,$ satisfies $u \in D(A) \cap L^\infty_+$. See e.g.\;\cite[sect.\,4]{LS} or \cite[sect.\,3]{KiS}.

By the Miyadera Perturbation Theorem, $\Lambda_r(a,\nabla a_n+ b_n)u=\Lambda_2(a,\nabla a_n+b_n)u =A u +(\nabla a_n+b_n)\cdot \nabla u$. By \cite[Theorem 2.1]{LS} (see also \cite[Theorem G.1]{KiS}), $u, u^\frac{r}{2}, u^r \in D(A^\frac{1}{2}),$  and since $D(A^\frac{1}{2})=W^{1,2}$, $\nabla u^\frac{r}{2} = \frac{r}{2} u^{\frac{r}{2}-1}  \nabla u$.
Thus, we have
\[
-\big \langle \frac{d}{dt} u, u^{r-1} \big \rangle = \big \langle A u, u^{r-1} \big \rangle + \big \langle (\nabla a_n + b_n) \cdot \nabla u, u^{r-1} \big \rangle, 
\]
where $\big \langle A u, u^{r-1} \big \rangle := \langle \nabla u \cdot a \cdot\nabla u^{r-1}\rangle $,
\begin{align*}
-\frac{d}{dt} \| u \|^r_r  = \frac{4}{r^\prime} \big \langle \nabla u^\frac{r}{2} \cdot a \cdot \nabla u^\frac{r}{2} \big \rangle + \frac{2}{r}\big \langle (\nabla a_n + b_n) \cdot \nabla u^{\frac{r}{2}}, u^{\frac{r}{2}} \big \rangle.
\end{align*}
By \eqref{pos},
\begin{equation}
\label{ineq9}
\tag{$\circ$}
-\frac{d}{dt} \| u \|^r_r  \geq \frac{4}{r^\prime} \big \langle \nabla u^\frac{r}{2} \cdot a \cdot \nabla u^\frac{r}{2}\big \rangle + \frac{2}{r}\big \langle b_n \cdot \nabla u^{\frac{r}{2}},u^{\frac{r}{2}} \big \rangle.
\end{equation}
Using the conditions $r \geq \frac{2}{2-\sqrt{\delta_1}}$, $b_n \in \mathbf{F}_{\delta_1}(A)$ and completing the quadratic estimate
\begin{align*}
2 |\big \langle u^\frac{r}{2} b_n \cdot \nabla u^\frac{r}{2} \big \rangle | & \leq \alpha \|b_a u^\frac{r}{2} \|_2^2 + \alpha^{-1} \| A^\frac{1}{2} u^\frac{r}{2} \|^2_2 \\
& \leq (\alpha \delta_1 + \alpha^{-1} )\big\langle (\nabla u^\frac{r}{2})\cdot a \cdot (\nabla u^\frac{r}{2}) \big\rangle  +\alpha \lambda\delta_1 \|u\|_r^r,
 \end{align*}
we obtain (choosing $\alpha = \frac{r^\prime}{2}$ and taking into account that $\sqrt{\delta_1} \leq \frac{2}{r^\prime}$ for $r \geq \frac{2}{2-\sqrt{\delta_1}} $)
$$
2 |\big \langle u^\frac{r}{2} b_n \cdot \nabla u^\frac{r}{2} \big \rangle | \leq \frac{4}{r^\prime} \big\langle (\nabla u^\frac{r}{2})\cdot a \cdot (\nabla u^\frac{r}{2}) \big\rangle + \frac{\lambda\delta_1r^\prime}{2} \|u\|_r^r.
$$
The previous estimate applied in \eqref{ineq9} yields
\begin{equation*}
\|e^{-t\Lambda_r(a,\nabla a_n + b_n)}\|_{r \rightarrow r} \leq e^{\omega_rt}, \quad \omega_r=\frac{\lambda\delta_1}{2(r-1)}, \quad r \geq \frac{2}{2-\sqrt{\delta_1}}.
\end{equation*}
In particular, $\|u_n\|_\infty \leq \|f\|_\infty$.

2)~Fix $\kappa > 2.$ Define
\begin{equation}
\label{zeta1}
\eta(t):=\left\{
\begin{array}{ll}
0, & \text{ if } t< \kappa, \\
\big( \frac{t}{\kappa} - 1 \big)^\kappa, & \text{ if }\kappa \leq t \leq 2 \kappa, \quad \quad \text{ and } \zeta(x) = \eta(\frac{|x|}{R}), \;\; R > 0.\\
1, & \text{ if } 2 \kappa < t,
\end{array}
\right.
\end{equation}
Note that $ |\nabla \zeta | \leq R^{-1} \mathbf{1}_{\nabla \zeta} \zeta^{1-\frac{1}{\kappa}}.$
Let $u \equiv u_n$ be as above, $v := \zeta u_n \geq 0.$  
Clearly,
\[
\left\langle \zeta \bigg(\frac{d}{d t} + A + (\nabla a_n + b_n \bigg) \cdot\nabla ) u_n, v^{r-1} \right\rangle =0.
\]
Since $v, v^\frac{r}{2}, \zeta v^{r-1} \in W^{1,2}$, it is easy to justify the following equation and equality ($[F, G]_- := F G - G F$):
\begin{align*}
\bigg\langle \bigg(\frac{d}{d t} + A + (\nabla a_n + b_n) \cdot \nabla\bigg) v, v^{r-1}\bigg \rangle & = \big\langle [A, \zeta]_-u_n + u_n (\nabla a_n+b_n) \cdot \nabla \zeta, v^{r-1} \big\rangle 
\tag{$\star$} \label{star2} \\
\langle [A, \zeta]_-u_n, v^{r-1} \rangle & =\frac{2}{r^\prime} \big\langle \nabla v^\frac{r}{2} \cdot a  \cdot \nabla \zeta, u_nv^{\frac{r}{2}-1}  \big\rangle - \langle \nabla \zeta \cdot a  \cdot \nabla u_n, v^{r-1} \rangle \\
& = \frac{2}{r^\prime} \big\langle \nabla v^\frac{r}{2} \cdot \frac{a}{\zeta} \cdot \nabla \zeta, v^\frac{r}{2} \big\rangle - \frac{2}{r}\big\langle \nabla \zeta \cdot \frac{a}{\zeta} \cdot \nabla v^\frac{r}{2}, v^\frac{r}{2} \big\rangle + \big\langle \nabla \zeta \cdot \frac{a}{\zeta^2} \cdot \nabla \zeta, v^r \big\rangle.
\end{align*}
In the LHS of \eqref{star2}, we argue as in 1) in order to get rid of the term $\langle \nabla a_n \cdot \nabla v, v^{r-1} \rangle$  and to estimate the term $\langle b_n \cdot \nabla v, v^{r-1} \rangle$. By Lemma \ref{nabla_a_lem}, we have $\nabla a_n + b_n \in \mathbf{F}_{\delta_3}(A)$, $\sqrt{\delta_3}=2\frac{d-1}{d-2}\frac{|c|}{1+c}+\sqrt{\delta_1}$. Therefore, by the quadratic estimates 
\begin{align*}
\big\langle u_n (\nabla a_n+b_n) \cdot \nabla \zeta, v^{r-1} \big\rangle & = \big\langle (\nabla a_n+b_n) \cdot \frac{\nabla \zeta}{\zeta}, v^r \big\rangle \\
& \leq \frac{\epsilon \sqrt{\delta_3}}{r} \big \langle \nabla v^\frac{r}{2}\cdot a \cdot  \nabla v^\frac{r}{2} \big \rangle + \frac{r \sqrt{\delta_3}}{4 \epsilon}\big\langle \nabla \zeta \cdot \frac{a}{\zeta^2}\cdot \nabla \zeta, v^r \big\rangle + \frac{\epsilon \lambda\delta_1}{r \sqrt{\delta_3}} \|v\|_r^r \;\;( \epsilon > 0 ),\\
\frac{2}{r}(r-2)\big\langle \nabla \zeta \cdot \frac{a}{\zeta} \cdot \nabla v^\frac{r}{2}, v^\frac{r}{2} \big\rangle &\leq\frac{\epsilon \sqrt{\delta_3}}{r} \big \langle \nabla v^\frac{r}{2}\cdot a \cdot  \nabla v^\frac{r}{2} \big \rangle + \frac{(r-2)^2}{r \epsilon \sqrt{\delta_3}}\big\langle \nabla \zeta \cdot \frac{a}{\zeta^2}\cdot \nabla \zeta, v^r \big\rangle,
\end{align*}
we get from \eqref{star2}
\[
\frac{d}{d t} \|v\|_r^r + 2\bigg( \frac{2}{r^\prime} - \sqrt{\delta_1} - \epsilon \sqrt{\delta_3} \bigg) \big \langle \nabla u^\frac{r}{2}\cdot a \cdot \nabla u^\frac{r}{2} \big \rangle \leq \bigg(\frac{(r-2)^2}{\epsilon \sqrt{\delta_3}} + \frac{r^2 \sqrt{\delta_3}}{4 \epsilon} +r \bigg)\big\langle \nabla \zeta \cdot \frac{a}{\zeta^2}\cdot \nabla \zeta, v^r \big\rangle + \frac{r + \epsilon}{\sqrt{\delta_3}} \lambda\delta_1 \|v\|_r^r.
\]
Recalling that $\frac{2}{r^\prime} > \sqrt{\delta_1},$ we can find $\epsilon$ so small that $\frac{2}{r^\prime} - \sqrt{\delta_1} - \epsilon \sqrt{\delta_3} \geq 0.$ Thus
\[
\frac{d}{d t} \|v\|_r^r \leq \bigg(\frac{4 (r-2)^2 + r^2 \delta_3}{4 \epsilon \sqrt{\delta_3}} +r \bigg)\big\langle \nabla \zeta \cdot \frac{a}{\zeta^2}\cdot \nabla \zeta, v^r \big\rangle + \frac{r + \epsilon}{\sqrt{\delta_3}} \lambda\delta_1 \|v\|_r^r. \tag{$\star\star$}
\]
Next, $\big\langle \nabla \zeta \cdot \frac{a}{\zeta^2}\cdot \nabla \zeta, v^r \big\rangle \leq \xi R^{-2} \| \mathbf 1_{\nabla \zeta} \zeta^{-2 \theta} v^r \|_1,$ where $\theta = \kappa^{-1}$ and $\mathbf 1_{\nabla \zeta}$ denotes the indicator of the support of $|\nabla \zeta|.$ Since $\|u_n\|_\infty \leq \|f\|_\infty, \; \|\mathbf 1_{\nabla \zeta}\|_\frac{r}{2\theta} \leq c(d,\theta) R^\frac{2\theta d}{r}$ and
\[
\| \mathbf 1_{\nabla \zeta} \zeta^{-2 \theta} v^r \|_1 \leq \|\mathbf 1_{\nabla \zeta}u_n^{2 \theta} \|_\frac{r}{2\theta} \|v\|_r^{r-2\theta}\leq \|\mathbf 1_{\nabla \zeta} \|_\frac{r}{2 \theta} \|u_n\|_\infty^{2\theta} \|v\|_r^{r-2\theta},
\]
we obtain, using the Young inequality, the crucial estimate
\[
\big\langle \nabla \zeta \cdot \frac{a}{\zeta^2}\cdot \nabla \zeta, v^r \big\rangle \leq \frac{2 \theta}{r} [\xi c(d)]^\frac{r}{2 \theta} R^{d-\frac{r}{\theta}} \|f\|_\infty^r + \frac{r-2\theta}{r} \|v\|_r^r.
\]
Fix $\theta$ by $0 < \theta < \frac{r}{d+2r}.$ Now, from $(\star\star)$ we obtain the inequality
\[
\frac{d}{d t} \|v\|_r^r \leq N(r, d, \delta_1, c) \|v\|_r^r + M(r, d, \delta_1, c) R^{-\gamma} \|f\|_\infty^r, \;\; \gamma = \frac{r}{\theta}-d > 0, \tag{$\star\star\star$}
\]
from which we conclude that, for given $T, f \in L^2 \cap L^\infty_+, \;\varepsilon > 0,$ there exists $R$ such that $$\sup_{t \in [0,T], n \geq 1} \|\zeta u_n(t)\|_r \leq \varepsilon.$$

3) For $\kappa$, $R$ determined above set
\begin{equation}
\label{zeta2}
\eta(t):=\left\{
\begin{array}{ll}
1, & \text{ if } t< 2 \kappa, \\
\big( 1 - \frac{1}{\kappa} (t - 2 \kappa) \big)^\kappa, & \text{ if } 2 \kappa \leq t \leq 3 \kappa, \quad \quad \text{ and } \zeta(x) = \eta(\frac{|x|}{R}), \;\; R > 0.\\
0, & \text{ if } 3 \kappa < t,
\end{array}
\right.
\end{equation}
Note that $ |\nabla \zeta | \leq R^{-1} \mathbf{1}_{\nabla \zeta} \zeta^{1-\frac{1}{k}}.$
Set $g := u_n - u_m$ and $v := \zeta g.$ Then, subtracting the equations for $u_n$ and $u_m$, multiplying the difference by $\zeta v|v|^{r-2}$ and integrating, we have
\[
\bigg \langle \zeta \bigg (\frac{d}{d t} + A +(\nabla a_n+b_n) \cdot \nabla \bigg ) g + \zeta (\nabla a_n - \nabla a_m + b_n - b_m) \cdot \nabla u_m, v |v|^{r-2} \bigg \rangle = 0,
 \]
\begin{align*}
& \bigg \langle \bigg (\frac{d}{d t} + A + (\nabla a_n+b_n) \cdot \nabla \bigg)v, v |v|^{r-2} \bigg \rangle \\
& = \big \langle [A, \zeta]_- g + v (\nabla a_n+b_n) \cdot \frac{\nabla \zeta}{\zeta}, v |v|^{r-2} \big \rangle + \langle \zeta (\nabla a_m - \nabla a_n + b_m - b_n) \cdot \nabla u_m, v |v|^{r-2} \rangle,
\end{align*}
where
 \[
\langle [A, \zeta]_- g , v |v|^{r-2} \rangle = \frac{2(r-2)}{r} \big\langle \nabla |v|^\frac{r}{2} \cdot \frac{a}{\zeta} \cdot \nabla \zeta, |v|^\frac{r}{2} \big\rangle + \big\langle \nabla \zeta \cdot \frac{a}{\zeta^2} \cdot \nabla \zeta, |v|^r \big\rangle,
\]
so arguing as in 2) we have
\begin{align*}
\frac{d}{d t} \|v\|_r^r & \leq \bigg(\frac{4 (r-2)^2 + r^2 \delta_3}{4 \epsilon \sqrt{\delta_3}} +r \bigg)\big\langle \nabla \zeta \cdot \frac{a}{\zeta^2}\cdot \nabla \zeta, v^r \big\rangle + \frac{r + \epsilon}{\sqrt{\delta_3}} \lambda\delta_1 \|v\|_r^r\\
& + r \langle \zeta (\tilde{b}_m-\tilde{b}_n) \cdot \nabla u_m, v |v|^{r-2} \rangle \text{ with the same } \epsilon \text{ as in } (\star \star), 
\end{align*}
where $\tilde{b}_m:=\nabla a_m + b_m \in \mathbf{F}_{\delta_3}(A)$,
\begin{align*}
\langle \zeta (\tilde{b}_m-\tilde{b}_n) \cdot \nabla u_m, v |v|^{r-2} \rangle & \leq \langle \zeta (\tilde{b}_m-\tilde{b}_n) \cdot a^{-1} \cdot (\tilde{b}_m-\tilde{b}_n) \rangle^\frac{1}{2} \langle \nabla u_m \cdot a \zeta \cdot \nabla u_m \rangle^\frac{1}{2} (2\|f\|_\infty)^{r-1}.
\end{align*}

In order to estimate $\int_0^T \langle \nabla u_m(t) \cdot a \zeta \cdot \nabla u_m (t) \rangle d t$ note that $\langle \frac{d}{d t} u_m + A u_m + \tilde{b}_m \cdot \nabla u_m, \zeta u_m \rangle = 0,$ or
\[
\frac{1}{2} \frac{d}{d t} \langle \zeta u_m^2 \rangle + \langle \nabla u_m \cdot a \zeta \cdot \nabla u_m \rangle + \langle \nabla u_m \cdot a u_m \cdot \nabla \zeta \rangle + \langle \tilde{b}_m \cdot \nabla u_m, \zeta u_m \rangle = 0,
\]
and so
\[
\frac{d}{d t} \langle \zeta u_m^2 \rangle + \langle \nabla u_m \cdot a \zeta \cdot \nabla u_m \rangle \leq 2 \big(\big \langle \nabla \zeta \cdot \frac{a}{\zeta} \cdot \nabla \zeta \big \rangle + \langle \zeta \tilde{b}_m \cdot a^{-1} \cdot \tilde{b}_m \rangle \big) \|f\|^2_\infty,
\]
\begin{align*}
\int_0^T \langle \nabla u_m(t) \cdot a \zeta \cdot \nabla u_m (t) \rangle d t & \leq \| f \|_2^2  + 2 T \big( \big\langle \nabla \zeta \cdot \frac{a}{\zeta} \cdot \nabla \zeta \big \rangle + \langle \zeta \tilde{b}_m \cdot a^{-1} \cdot \tilde{b}_m \rangle \big) \|f\|^2_\infty \\
& \equiv \| f \|_2^2  + T L(R) \|f\|_\infty^2.
\end{align*}
Now it should be clear that the above is sufficient for concluding that the following inequality analogous to $(\star\star\star)$  holds for all $n, m$ and $t \in [0,T],$
\begin{align*}
e^{-N t} \|\zeta g\|_r^r & \leq t M R^{-\gamma} \|f\|_\infty^r \\ 
& + t (2 \|f\|_\infty)^{r-1} \big(\|f\|_2^2 + t L(R) \|f\|_\infty^2 \big)^\frac{1}{2} \big\langle \zeta (\tilde{b}_n-\tilde{b}_m)\cdot a^{-1}\cdot(\tilde{b}_n-\tilde{b}_m) \rangle^\frac{1}{2}. 
\end{align*}
 %\begin{align*}
%& \big \langle \big (\partial_t + A +(\nabla a_n+b_n) \cdot \nabla \big)v, v |v|^{r-2} \big \rangle \\
%&= \big \langle [A, \zeta]_- g + v (\nabla a_n+b_n) \cdot \frac{\nabla \zeta}{\zeta}, v |v|^{r-2} \big \rangle + \langle \zeta (b_m - b_n+\nabla a_m-\nabla a_n) \cdot \nabla u_m, v |v|^{r-2} \rangle.
 %\end{align*}
%We get rid of $\nabla a_n$ in the LHS and in the first term in RHS as in 1. 
%We show that the second term in the RHS converges to $0$ as $n,m \rightarrow \infty$ as in the proof of \cite[Theorem 3.2]{KiS}, using that $b_n \rightarrow b$, $\nabla a_n \rightarrow \nabla a$ in $L^{2}_{\loc}$.

4)\;Combining 2) and 3) and using that $b_n \rightarrow b$, $\nabla a_n \rightarrow \nabla a$ strongly in $L^{2}_{\loc}$, we obtain that for each $0 <T < \infty$, $r > \frac{2}{2-\sqrt{\delta_1}}$ we can find $R < \infty$ and $M < \infty$ such that
\[
\sup_{t \in [0, T], \; n, m \geq 1} \|(1-\mathbf{1}_{B(o, 2 k R)})(u_n(t) - u_m(t))\|_r < \epsilon, \quad
 \sup_{t \in [0, T], n, m \geq M}\|\mathbf{1}_{B(o, 2 k R)}(u_n(t) - u_m(t))\|_r < \epsilon,
\]
which yields the required. The proof of Theorem \ref{thm:A1} for $-1<c<0$, $r \geq 2$ is completed.

\smallskip

If $c>0$, then the estimate \eqref{pos} is clearly not valid, and we have to get rid of the term $\langle \nabla a_n \cdot \nabla v, v^{r-1} \rangle$ using the quadratic estimates only (replacing above $\delta_1$ by  $\delta_2>\delta_1$).
\end{proof}

\begin{theorem}
\label{conv_appendix}
Let $d \geq 3$.
Let $a=I+c|x|^{-2}x \otimes x$, $
a^\varepsilon:=I+c|x|^{-2}_\varepsilon x \otimes x,
$ $|x|_\varepsilon^2:=|x|^2+\varepsilon$, $\varepsilon>0$. 
Set $a_n:=a^{\varepsilon_n}$, $\varepsilon_n \downarrow 0$. Let $b \in \mathbf{F}_\delta$, let $b_n$'s be as in Theorem \ref{thm:A1}.

{\rm (\textit{i})} Assume that $q \geq d-2$, $d \geq 4$, $c$, $\delta$ satisfy the assumptions of Theorem \ref{thm2}(\textit{ii}), or $q=2$, $d \geq 3$, $c$, $\delta$ satisfy the assumptions of Theorem \ref{thm2}(\textit{iii}).
Then $(\mu+\Lambda_q(a,b))^{-1}$, $(\mu+\Lambda_q(a_n,b_n))^{-1}$, $\mu>\omega_q$, are well defined, and 
\[
(\mu+\Lambda_q(a,b))^{-1} = s\mbox{-}L^q\mbox{-}\lim_n (\mu+\Lambda_q(a_n,b_n))^{-1}.
\]
Here $\Lambda_q(a_n, b_n)=-\nabla \cdot a_n \cdot \nabla + b_n \cdot \nabla$, $D(\Lambda_q(a_n, b_n))=W^{2,q}$.

{\rm (\textit{ii})} Assume that $q \geq d-2$, $d \geq 4$, $c$, $\delta$ satisfy the assumptions of Theorem \ref{thm4}(\textit{ii}), or $q=2$, $d \geq 3$,  $c<0$, $\delta$ satisfy the assumptions of Theorem \ref{thm4}(\textit{iii}). Then 
$(\mu+\Lambda_q(a,\nabla a + b))^{-1}$, $(\mu+\Lambda_q(a_n,\nabla a_n+b_n))^{-1}$, $\mu>\omega_q$, are well defined, and 
\[
(\mu+\Lambda_q(a,\nabla a+ b))^{-1} = s\mbox{-}L^q\mbox{-}\lim_n (\mu+\Lambda_q(a_n,\nabla a_n + b_n))^{-1}.
\]
\end{theorem}

\begin{proof}
We modify the proof of Theorem \ref{thm:A1} but will work with resolvents instead of semigroups.

To prove (\textit{i}), set $\delta_*:=[1 \vee (1+c)^{-2}]\delta$, $\tilde{b}:=b$, $\tilde{b}_n:=b_n$.

To prove (\textit{ii}), set
$$
\sqrt{\delta_*}:=\left\{
\begin{array}{ll}
\sqrt{\delta}+2 \frac{d-1}{d-2}\frac{c}{1+c}, & c>0, \\
(1+c)^{-1}\sqrt{\delta},  & -1<c<0.
\end{array}
\right.
$$
and $\tilde{b}:=b + \nabla a$, $\tilde{b}_n:=b_n + \nabla a_n$.

\smallskip

\textbf{1.~}First, prove (\textit{i}) and (\textit{ii}) for $c>0$. Set $A^n \equiv [-\nabla \cdot a_n \cdot \nabla\upharpoonright C_c^\infty]^{\rm clos}_{2 \to 2}.$ Then $\tilde{b} \in \mathbf{F}_{\delta_*}(A)$, $\tilde{b}_n \in \mathbf{F}_{\delta_*}(A^n)$ (for details see the proofs of Theorems \ref{thm2} and \ref{thm4}, respectively), where, by our assumptions, $\delta_*<4$.
Therefore, by Theorem \ref{thm:A1}, $(\mu+\Lambda_q(a,\tilde{b}))^{-1}$, $(\mu+\Lambda_q(a_n,\tilde{b}_n))^{-1}$, $q>\frac{2}{2-\sqrt{\delta_*}}$, $\mu>\omega_q$, are well defined, and 
$
\lim_n\|(\mu+\Lambda_q(a,\tilde{b}))^{-1}f - (\mu+\Lambda_q(a,\tilde{b}_n))^{-1}f\|_q = 0$, $f \in L^q$.
Thus, it suffices to show that 
$$
\lim_n\|(\mu+\Lambda_q(a,\tilde{b}_n))^{-1}f - (\mu+\Lambda_q(a_n,\tilde{b}_n))^{-1}f\|_q =0.
$$

(a)~Fix $f \in L^\infty \cap L^2_+$. Set $u_n:=(\mu+\Lambda_q(a,\tilde{b}_n))^{-1}f \geq 0$, $\tilde{u}_{n}:=(\mu+\Lambda_q(a_n,\tilde{b}_n))^{-1}f \geq 0$. 
Let $v := \zeta u_n \geq 0$, where $\zeta=\zeta(R)$, $R>0$, is defined by \eqref{zeta1}. Note that $\langle (\mu+ \Lambda_q(a,\tilde{b}_n))u_n,\zeta v^{q-1}\rangle = \langle  (\mu + A + \tilde{b}_n \cdot\nabla ) u_n, \zeta v^{q-1} \rangle $ according to step 2 in the proof of Theorem \ref{thm:A1}, and hence
\[
\langle \zeta (\mu + A + \tilde{b}_n \cdot\nabla ) u_n, v^{q-1} \rangle =\langle  \zeta f, v^{q-1} \rangle.
\]
Now, proceeding as in step 2 of the proof of Theorem \ref{thm:A1}, we arrive at the following.
For every $\varepsilon>0$ there exists $R>0$ such that
\begin{equation*}
\|\zeta u_n\|_q \leq \varepsilon,  \quad n \geq 1, \quad \mu>\omega_q.
\end{equation*}
Similarly,
\begin{equation*}
\|\zeta \tilde{u}_n\|_q \leq \varepsilon \quad n \geq 1, \quad \mu>\omega_q.
\end{equation*}

(b)\;Set $g_n := u_n - \tilde{u}_{n}$. For the $R$ determined above, 
set
$v := \zeta g_n$, where $\zeta=\zeta(R)$ is defined by \eqref{zeta2}. Subtracting the equations for $u_n$ and $\tilde{u}_n$, we have
$$
\big \langle \big (\mu + A +\tilde{b}_n \cdot \nabla \big ) g_n - \nabla \cdot (a-a_n) \cdot \nabla \tilde{u}_n, \zeta v |v|^{q-2} \big \rangle = 0.
$$
Arguing as in step 3 of the proof of Theorem \ref{thm:A1}, we arrive at the inequality
\begin{align*}
\mu \|v\|_q^q & \leq M(q, d, \delta_*) R^{-\gamma} \|f\|_\infty^q 
 +   |\langle \zeta \nabla \cdot (a-a_n) \cdot \nabla \tilde{u}_n,v|v|^{q-2}\rangle|.
\end{align*}
To show that $\zeta(u_n - \tilde{u}_{n}) \rightarrow 0$ strongly in $L^q$ as $n \rightarrow 0$, it remains to prove that $\lim_n |Z|=0$, where $Z:=\langle (a - a_n) \cdot \nabla \tilde{u}_{n}, \nabla(\zeta v |v|^{q-2}) \rangle$. The latter is possible due to the bounds $\|\nabla u_n\|_{\frac{qd}{d-2}} \leq K\|f\|_q$, $\|\nabla \tilde{u}_n\|_{\frac{qd}{d-2}} \leq K\|f\|_q$ (cf.\,the proof of Theorem \ref{thm2} (steps 1-3) for (\textit{i}), the proof of Theorem \ref{thm4} (steps 1-3) for (\textit{ii})). Indeed,
\begin{align*}
Z&=
q\langle \zeta^{q-1}\nabla \zeta\cdot(a-a_n)\cdot\nabla \tilde{u}_n,g_n|g_n|^{q-2}\rangle+ \langle \nabla g_n\cdot(a-a_n)\cdot\nabla \tilde{u}_n,\zeta^q|g_n|^{q-2}\rangle \\
&+(q-3)\langle \nabla |g_n|\cdot(a-a_n)\cdot\nabla \tilde{u}_n,\zeta^q g_n|g_n|^{q-3}\rangle \\
& \equiv qZ_1 + Z_2 +(q-3)Z_3,
\end{align*}
and ($q_*:=\frac{qd}{d-2}>2$)
\begin{align*}
|Z_1|\leq & q \|\nabla \zeta\cdot (a-a_n)\|_{q_*^\prime} \|\nabla \tilde{u}_n\|_{q_*} \|g_n\|_\infty^{q-1},\\
|Z_2|\leq &\|\nabla g_n \cdot \zeta(a-a_n)\cdot\nabla \tilde{u}_n\|_1 \|g_n\|_\infty^{q-2},\\
|Z_3|\leq &\|\nabla |g_n| \cdot \zeta(a-a_n)\cdot\nabla \tilde{u}_n\|_1 \|g_n\|_\infty^{q-2},\\
\|\nabla |g_n|\|_{q_*} \leq & \|\nabla g_n\|_{q_*} \leq 2K\|f\|_q, \;\; \sup_{i,j}\|\zeta(a_{ij}-a_{ij}^{\varepsilon_n})\|_\frac{q_*}{q_*-2} \rightarrow 0 \text{ as } n\to 0.
\end{align*}
Thus, by H\"older's inequality and $|a_{ij}-a_{ij}^\varepsilon|\leq c  \varepsilon |x|_\varepsilon^{-2} \downarrow 0$ a.e.\,as $n \rightarrow \infty$,
\[
\lim_n |Z|= 0.
\]
It follows that $\zeta(u_n - \tilde{u}_{n}) \rightarrow 0$ in $L^q$.

Combining the results of (a) and (b), we obtain the required.

\smallskip

\textbf{2.~}To prove (\textit{ii}) with $c<0$, we repeat the proof above but taking into account the bound \eqref{pos} in the proof of Theorem \ref{thm:A1}.
\end{proof}

\end{document}